\documentclass[11pt, a4paper]{article}
\usepackage{amssymb,amsmath,amsthm}
\usepackage{fancyhdr}
\usepackage[british]{babel}
\usepackage{enumitem}

\usepackage[usenames,dvipsnames,table]{xcolor}
\usepackage{tikz}
\usetikzlibrary{decorations.pathreplacing}
\usepackage{bbm}
\newcounter{propcounter}

\usepackage[margin=2.5cm]{geometry}

\newtheorem{theorem}{Theorem}[section]
\newtheorem{prop}[theorem]{Proposition}
\newtheorem{lemma}[theorem]{Lemma}
\newtheorem{cor}[theorem]{Corollary}

\theoremstyle{definition}
\newtheorem{problem}{Problem}

\newtheorem{defin}[theorem]{Definition}
\newtheorem{claim}[theorem]{Claim}
\newtheorem{step}{Step}

\newenvironment{poc}{\begin{proof}[Proof of claim]}{\end{proof}}

\newcommand{\xc}{\overline{X}}
\newcommand{\eps}{\varepsilon}

\newcommand{\cF}{\mathcal{F}}
\newcommand{\cQ}{\mathcal{Q}}

\newcommand{\sfG}{\mathsf{G}}
\newcommand{\sfT}{\mathsf{T}}

\newcommand*{\bfrac}[2]{\genfrac{(}{)}{}{}{#1}{#2}}
\newcommand{\Int}{\mathsf{Int}}
\newcommand{\Ext}{\mathsf{Ext}}
\newcommand{\Cen}{\mathsf{Ctr}}
\newcommand*{\abs}[1]{\lvert#1\rvert}
\newcommand*{\bigabs}[1]{\bigl|#1\bigr|}
\newcommand*{\Bigabs}[1]{\Bigl|#1\Bigr|}

\title{Extremal density for sparse minors and subdivisions}
\author{John Haslegrave\thanks{J.H. and H.L. were supported by the UK Research and Innovation Future Leaders Fellowship MR/S016325/1.} \and 
		Jaehoon Kim\thanks{J.K. was supported by the POSCO Science Fellowship of POSCO TJ Park Foundation and by the KAIX Challenge program of KAIST Advanced Institute for Science-X} \and 
		Hong Liu\footnotemark[1]}

\definecolor{darkblue}{rgb}{0,0,0.5}
\begin{document}
\maketitle
\begin{abstract}
	We prove an asymptotically tight bound on the extremal density guaranteeing subdivisions of bounded-degree bipartite graphs with a mild separability condition. As corollaries, we answer several questions of Reed and Wood on embedding sparse minors. Among others,
	\begin{itemize}
		\item $(1+o(1))t^2$ average degree is sufficient to force the $t\times t$ grid as a topological minor;
		
		\item $(3/2+o(1))t$ average degree forces \emph{every} $t$-vertex planar graph as a minor, and the constant $3/2$ is optimal, furthermore, surprisingly, the value is the same for $t$-vertex graphs embeddable on any fixed surface;
		
		\item a universal bound of $(2+o(1))t$ on average degree forcing \emph{every} $t$-vertex graph in \emph{any} nontrivial minor-closed family as a minor, and the constant 2 is best possible by considering graphs with given treewidth.
	\end{itemize}
\end{abstract}

\section{Introduction}

Classical extremal graph theory studies sufficient conditions forcing the appearance of substructures. A seminal result of this type is the Erd\H{o}s--Stone--Simonovits theorem~\cite{ESto,ESim}, determining the asymptotics of the average degree needed for subgraph containment. It reads as 
\[d_{\supseteq}(H):=\lim_{n\rightarrow \infty}\inf\{c:~\abs{G}\ge n~\text{ and }~d(G)\geq c\abs G\Rightarrow G\supseteq H\}=1-\frac{1}{\chi(H)-1},\]
where $\chi(H)$ is the chromatic number of $H$. We are interested here in the analogous problem of average degree conditions forcing $H$ as a minor. A graph $H$ is a \emph{minor} of $G$, denoted by $G\succ H$, if it can be obtained from $G$ by vertex deletions, edge deletions and contractions.

The study of such problem has a long history. An initial motivation was Hadwiger's conjecture that every graph of chromatic number $t$ has $K_t$ as a minor, which is a far-reaching generalisation of the four-colour theorem. Since every graph of chromatic number $k$ contains a subgraph of average degree at least $k-1$, a natural angle of attack is to find bounds on the average degree which will guarantee a $K_t$-minor. The first upper bound for general $t$ was given by Mader \cite{Mader1, Mader2}, who subsequently improved this bound to $O(t\log t)$. In celebrated work of Kostochka \cite{Kos84} and, independently, Thomason \cite{Thom84}, it was improved to the best possible bound $O(t\sqrt{\log t})$, Thomason subsequently determining the asymptotic \cite{Thom01}. Denoting
\[d_{\succ}(H):=\inf\{c:~d(G)\geq c\Rightarrow G\succ H\},\]
he proved that $d_{\succ}(K_t)=(\alpha+o_t(1))t\sqrt{\log t}$, where $\alpha=0.6382\ldots$ is explicitly defined. This remained the best order of magnitude bound even for the chromatic number question until very recent breakthrough by Norin, Postle and Song~\cite{NPS19} and by Postle~\cite{Pos20}.

For more general minors, Myers and Thomason \cite{MT05} resolved the problem when $H$ is polynomially dense, i.e.\ having $\abs{H}^{1+\Omega(1)}$ edges, showing that $d_{\succ}(H)=(\alpha\gamma(H)+o(1))\abs{H}\sqrt{\log \abs{H}}$ for $\alpha$ as above and some explicitly defined $\gamma(H)$. However, for sparse graphs their results only give $d_{\succ}(H)=o(\abs{H}\sqrt{\log \abs{H}})$, similar to the way that the Erd\H{o}s--Stone--Simonovits theorem only gives $d_{\supseteq}(H)=0$ for bipartite $H$, and so it is natural to ask for stronger bounds in this regime.

Reed and Wood \cite{RW16} considered sparser graphs, and in particular showed that for sufficiently large average degree $d(H)$ we have $d_{\succ}(H)<3.895 \abs{H}\sqrt{\log d(H)}$. They also obtain bounds linear in $e(H)$, which are better in the very sparse case of bounded average degree. Reed and Wood asked several interesting questions about the asymptotics of $d_{\succ}(H)$ for sparse $H$. Among sparse graphs, grids play a central role in graph minor theory \cite{DemHaj, RS86,RST94}. Indeed, Reed and Wood raised the question of determining $d_{\succ}(\sfG_{t,t})$,  where $\sfG_{t,t}$ is the $t\times t$ grid. That is, what is the minimum $\beta>0$ such that every graph with average degree at least $\beta t^2$ contains $\sfG_{t,t}$ as a minor. Trivially $\beta\geq 1$ in order for the graph to have enough vertices, and their results give a bound of $\beta\leq 6.929$. 

This question provides the motivating example for our results. However, we shall focus on a special class of minors: subdivisions or topological minors. A \emph{subdivision} of $H$ is a graph obtained from subdividing edges of $H$ to pairwise internally disjoint paths. The name of topological minor comes from its key role in topological graph theory. A cornerstone result in this area is Kuratowski’s theorem from 1930 that a graph is planar if and only if it does not contain a subdivision of $K_5$ or $K_{3,3}$. 
Again it is natural to ask what average degree $d$ will force $K_t$ as a topological minor, and we define analogously 
\[d_{\sfT}(H):=\inf\{c:~d(G)\geq c\Rightarrow G~\text{ contains }~H~\text{ as a topological minor}\}.\]
Clearly, for any $H$, $d_{\succ}(H)\le d_{\sfT}(H)$. By considering complete bipartite graphs it is easy to see that $d_{\sfT}(K_t)=\Omega(t^2)$. Koml\'os and Szemer\'edi \cite{K-Sz96} and, independently, Bollob\'as and Thomason \cite{BT98} gave a matching upper bound of $d_{\sfT}(K_t)=O(t^2)$. Note that clique topological minors are harder to guarantee than clique minors, as evidenced by the significant gap between this result and that of Kostochka and of Thomason for the latter. Furthermore, the leading coefficient of the quadratic bound on $d_{\sfT}(K_t)$ is still unknown. Much less is known for bounds on average degree guaranteeing sparse graphs as topological minors.

\subsection{Main result} We study in this paper the problem of finding subdivisions of a natural class of sparse bipartite graphs. In particular, our main result offers the asymptotics of the average degree needed to force subdivisions of such graphs, showing that a necessary bound is already sufficient. It reads as follows.

\begin{theorem}\label{thm: main}
For given $\eps>0$ and $\Delta\in \mathbb{N}$, there exist $\alpha_0$ and $d_0$ satisfying the following for all $0<\alpha<\alpha_0$ and $d\geq d_0$.
If $H$ is an $\alpha$-separable bipartite graph with at most $(1-\eps)d$ vertices and $\Delta(H)\leq \Delta$, and $G$ is a graph with average degree at least $d$, then $G$ contains a subdivision of $H$.
\end{theorem}

Here a graph $H$ is $\alpha$\textit{-separable} if there exists a set $S$ of at most $\alpha \abs{H}$ vertices such that every component of $H-S$ has at most $\alpha \abs{H}$ vertices. Graphs in many well-known classes are $o(1)$-separable. For example, any large graphs in nontrivial minor-closed family of graphs is $o(1)$-separable \cite{AST90,MRS16}.

As an immediate corollary, our main result answers the above question of Reed and Wood in a strong sense by showing that any $\beta>1$ is sufficient to force the $k$-dimensional grid $\sfG_{t,\ldots,t}^k$ not only as a minor but as a topological minor, and so
\[d_{\sfT}(\sfG_{t,\ldots,t}^k)=d_{\succ}(\sfG_{t,\ldots,t}^k)=(1+o_t(1))t^k.\]

We remark that the optimal constant $1$ in Theorem~\ref{thm: main} is no longer sufficient if $H$ is not bipartite. Indeed, if e.g.\ $H$ is the disjoint union of triangles, then the Corr\'adi--Hajnal theorem~\cite{CH63} implies that $d_{\succ}(H)=\frac{4}{3}\abs{H}-2$. We shall elaborate more on the leading constant in the next section.

Our proof utilises both pseudorandomness from Szemer\'edi's regularity lemma and expansions for sparse graphs. The particular expander that we shall make use of is an extension of the one introduced by Koml\'os and Szemer\'edi, which has played an important role in some recent developments on sparse graph embedding problems, see e.g.~\cite{LM17,KLShS17}.

\medskip

\noindent\textbf{Organisation.} Applications of our main result will be given in Section~\ref{sec: application}. Preliminaries on expanders and basic building blocks are given in Section~\ref{sec: prelim}. In Section~\ref{sec: outline}, we outline the strategies for proving Theorem~\ref{thm: main}. Its proof is then given in Sections~\ref{sec: dense}--\ref{sec-sparse}. Lastly, concluding remarks and some open problems are given in Section~\ref{sec: rmk}.

\section{Applications}\label{sec: application}

Reed and Wood~\cite{RW16} raised several interesting questions on the average degree needed to force certain sparse graphs as minors. As corollaries of our main results, we answer all of these questions, and some others, with asymptotically optimal bounds.

\subsection{Planar graphs}

\begin{problem}[\cite{RW16}]What is the least constant $c>0$ such that every graph with average degree at least $ct$ contains every planar graph with $t$ vertices as a minor?\end{problem}

Since a planar graph has average degree at most $6$, their results imply that $c\leq 14.602$. We can deduce the asymptotic answer to their question; in fact, rather surprisingly, the value is the same for graphs of any fixed genus, not just planar graphs.

We find it convenient  to use the following notation. The above problem basically askes for the maximum of $d_{\succ}(F)$ over all graphs in some family $\mathcal{F}$ with at most $t$ vertices. Define \[d_{\succ}(\cF,t):=\inf\{c:~d(G)\geq c\Rightarrow G\succ H, ~\forall H\in\cF \text{ with } \abs{H}\leq t\}.\]

\begin{theorem}\label{planar}
For any $\eps>0$ and any $g\geq 0$ there exists $t_0$ such that for every $t\geq t_0$ the following hold:
\begin{itemize}
\item every graph with average degree at least $(3/2+\eps)t$ contains every graph of genus at most $g$ with $t$ vertices as a minor; but
\item there exists a graph of average degree $3t/2-O(1)$ which does not contain every planar graph with $t$ vertices as a minor.
\end{itemize}
That is, writing $\cF_g$ for all graphs with genus $g$, we have $d_{\succ}(\cF_g,t)=\left(3/2+o(1)\right)t$.
\end{theorem}
\begin{proof}We first prove the second statement. Take any planar graph $H^*$ on $t$ vertices containing $\lfloor t/4\rfloor$ disjoint copies of $K_4$. It is easy to verify that $K_{2,n}$ is a series-parallel graph for any $n$, and so does not contain $K_4$ as a minor (see \cite{Duff}). Therefore any bipartite graph which does contain $K_4$ as a minor is not a subgraph of $K_{2,n}$ for any $n$, and so must have at least $3$ vertices in each part. If $G$ is a bipartite graph which contains $H^*$ as a minor, then $G$ must contain at least $3$ vertices in each part for each of the $\lfloor t/4\rfloor$ copies of $K_4$, so must contain at least $3\lfloor t/4\rfloor$ vertices in each part. Therefore the complete bipartite graph with parts of order $3\lfloor t/4\rfloor-1$ and $n$ does not contain $H^*$ as a minor. This graph has average degree at least $6\lfloor t/4\rfloor-3$ if $n$ is sufficiently large.

Next, we deduce the first statement from Theorem \ref{thm: main}. Let $\Delta = \lfloor20/\eps\rfloor$, and fix $\alpha\leq\alpha_0(\eps/2,\Delta)$. Let $H$ be an arbitrary $t$-vertex graph of genus at most $g$. If $H$ has a vertex of degree $k>\Delta$, replace it with two adjacent vertices of degrees $k-(\Delta-1)$ and $\Delta$; when doing this, allocate the neighbours of the original vertex to the two vertices in such a way as to preserve the genus. Continue until no vertices with degree bigger than $\Delta$ remain. Thus we obtain a graph $H'$ of genus at most $g$ with at most $t+ \sum_{v\in V(H)} \frac{d(v)}{\Delta-1} \leq t + (6t+O(1))/(\Delta-1) \leq (1+\eps/3) t$ vertices which contains $H$ as a minor.

If $g=0$ (i.e.\ $H'$ is planar) then take a four-colouring of $H'$, and set aside the two largest colour classes $A$ and $B$, which cover at least $(1+\eps/3)t/2$ vertices between them. If $g>0$, using a result of Djidjev \cite{Dji84} we may find a planar induced subgraph on $\abs{H'}-O(\sqrt{\abs{H'}})$ vertices. Thus, by extending a four-colouring of this subgraph, we may colour $H'$ in such a way that the two largest colour classes, $A$ and $B$, cover at least $\abs{H'}/2-o(t)$ vertices. So, in either case we have $\abs{V(H')\setminus (A\cup B)}\leq (1+\eps) t/2$ for $t$ sufficiently large.

We may assume that every other vertex has neighbours in both $A$ and $B$, since otherwise we could use a different colouring for which the total size of $A\cup B$ is larger. We define a bipartite graph $H''$ as follows. Vertices in $A\cup B$, and edges between them, are unchanged. For each vertex $v\in V(H')\setminus (A\cup B)$, $H''$ has two vertices $v_A,v_B$ with an edge between them. Every edge of $H'$ of the form $uv$ with $u\in A$ and $v\in V(H')\setminus (A\cup B)$ becomes an edge $uv_B$, and every edge of the form $uv$ with $u\in B$ and $v\in V(H')\setminus (A\cup B)$ becomes an edge $uv_A$. For every edge $vw$ with $v,w\in V(H')\setminus (A\cup B)$, choose an ordering $v,w$ arbitrarily and add the edge $v_Aw_B$. In the resulting graph $H''$, every vertex has degree at most that of the corresponding vertex in $H'$, and, by contracting every edge of the form $v_Av_B$, $H''$ contains $H'$ as a minor. 

The genus of $H''$ may be greater than $g$, but the bounded genus of $H'$ ensures that $H'$ is $\alpha/2$-separable for $t$ sufficiently large. Since any subset of $V(H')$ of size at most $\alpha\abs{H'}/2$ corresponds to a subset of $V(H'')$ of size at most $\alpha\abs{H''}$, $H''$ is $\alpha$-separable. Now $H''$ is a bipartite graph with at most $(3/2+\eps)t$ vertices, so by Theorem \ref{thm: main} we can find a subdivision of $H''$, which in turn contains $H$ as a minor, in any graph with average degree at least $(3/2+\eps)t$ provided $t$ is sufficiently large.
\end{proof}

\subsection{A universal bound for nontrivial minor-closed families}\label{sec-minor-closed}
Many important classes of graphs are naturally closed under taking minors, for example, graphs embeddable on a given surface considered in Theorem~\ref{planar}.
The seminal graph minor theorem of Robertson and Seymour (proved in a sequence of papers culminating in \cite{RS04}) shows that every minor-closed family can be characterised by a finite list of minimal forbidden minors. For example, the linklessly-embeddable graphs are defined by a minimal family of seven forbidden minors, including $K_6$ and the Petersen graph \cite{RST95}. The existence of a forbidden-minor characterisation has far-reaching algorithmic implications, since for any fixed graph $F$ there exists an algorithm to determine whether an $n$-vertex graph contains $F$ as a minor in $O(n^3)$ time \cite{RS95}, and hence there is a cubic-time algorithm (since improved to quadratic \cite{KKR}) to check for membership of any given minor-closed family; prior to these results it was not even known that the property of having a linkless embedding was decidable. However, the constants concealed by the asymptotic notation are typically prohibitively large. Furthermore, for many families a complete forbidden minor classification, and hence a specific algorithm, is not known, and the number of minimal forbidden minors can be extremely large, even for families that may be very naturally and simply defined. For example, there are over 68,000,000,000 minimal forbidden minors for the class of $Y\Delta Y$-reducible graphs \cite{Yu05}.

We can extend the methods of the previous subsection to minor-closed families more generally. For each $k\in \mathbb{N}$, define $\alpha_k(G):= \max\{ \abs{U}: U\subseteq V(G), \chi(G[U])=k\}$. So $\alpha_1(G)$ is the usual independence number and $\alpha_2(G)$ is the maximum size of the union of two independent sets.
\begin{theorem}\label{thm-minor-closed}
	Let $\cF$ be a nontrivial minor-closed family. For each $F\in \cF$ with $t$ vertices, we have 
	\[ 2t - 2\alpha(F) - O(1)  \leq d_{\succ}(F) \leq 2t - \alpha_2(F) + o(t).\]
\end{theorem}
\begin{proof}
It is well known that the $t$-vertex graphs in $\cF$ are $o(1)$-separable with at most $C_{\cF} t$ edges for some constant $C_{\cF}$ \cite{AST90}.

To prove the upper bound, take two disjoint independent sets $A$ and $B$ with $\alpha_2(F)=\abs{A\cup B}$.
By the same argument as in Theorem~\ref{planar}, we can define a bipartite graph $F''$ containing $F$ as a minor with $2t-\alpha_2(F) + o(t)$ vertices having bounded maximum degree. Apply Theorem~\ref{thm: main} to $F''$ to obtain the upper bound.

To prove the lower bound, consider $K_{s,n-s}$ where $s= t-\alpha(F)-1$. If it contains an $F$-minor, let $V_1,\dots, V_{\abs{F}}$ be the vertex sets corresponding to the vertices of $F$. By our choice of $s$, it is easy to see that at least $\alpha(F)+1$ of them have to completely reside in the independent set of size $n-s$ in $K_{s,n-s}$, which is impossible. Thus, $K_{s,n-s}$ does not contain an $F$-minor. By choosing large $n$, we have $d(K_{s,n-s}) \geq 2t-2\alpha(F)-O(1)$. 
\end{proof}

Theorem~\ref{thm-minor-closed} yields the following universal bound for all nontrivial minor closed families.
\begin{cor}\label{cor: }
	For any nontrivial minor closed family $\cF$, we have 
	$$d_{\succ}(\cF,t) \leq (2+o(1))t.$$ 
\end{cor}
We remark that the constant $2$ above cannot be improved as we shall see in Corollary~\ref{cor: tw-minor}.

\subsection{Graphs with given treewidth or clique minor}

Reed and Wood asked the following questions for specific minor-closed families.

\begin{problem}[\cite{RW16}]What is the least function $g_1$ such that every graph with average degree at least $g_1(k)\cdot t$ contains every graph with $t$ vertices and treewidth at most $k$ as a minor?\end{problem}

\begin{problem}[\cite{RW16}]What is the least function $g_2$ such that every graph with average degree at least $g_2(k)\cdot t$ contains every $K_k$-minor-free graph with $t$ vertices as a minor?\end{problem}

Graphs with treewidth at most $k$ are $k$-degenerate, and hence have at most $kt$ edges, and graphs without a $K_k$-minor have average degree $O(k\sqrt{\log k})$. Consequently the result of Reed and Wood~\cite{RW16} showed that $g_i(k)=O(\sqrt{\log k})$ for $i\in\{1,2\}$. As a corollary of Theorem~\ref{thm-minor-closed}, we get the following optimal bound of $g_i(k)=2+o_k(1)$, showing that somewhat surprisingly, when $k$ is sufficiently large, both the treewidth and the size of a largest clique minor play negligible roles in the leading coefficient.

\begin{cor}\label{cor: tw-minor}
	Every graph with average degree $(2+o_k(1))t$ contains every graph with $t$ vertices, which either has treewidth at most $k$ or is $K_k$-minor-free, as a minor.
\end{cor}
\begin{proof}
	The upper bound follows immediately from Theorem~\ref{thm-minor-closed}. Note that a disjoint union of copies of $K_{k+1}$ has treewidth $k$. Then the unbalanced complete bipartite graph $K_{(1-1/(k+1))t-1,n}$ provides the matching lower bound $2(1-1/(k+1))t=(2+o_k(1))t$.
	
	For graphs without $K_k$-minor, consider instead a disjoint union of copies of $K_{k-1}$.
\end{proof}

\subsection{Beyond minor-closed classes}

In Section \ref{sec-minor-closed}, the two properties of minor-closed families that we needed were $o(1)$-separability and bounded average degree. Many other sparse graph classes have these properties. In particular, any class which obeys a strongly-sublinear separator theorem is $o(1)$-separable, see \cite{MRS16}. 

A $k$-\textit{shallow minor} of a graph $G$ is a minor for which each contracted subgraph has radius at most $k$. We say that a graph class $\mathcal C$ has \textit{bounded expansion} if the average degree of $k$-shallow minors of graphs in $\mathcal C$ is bounded by a function of $k$; in particular, since $0$-shallow minors are just subgraphs, $\mathcal C$ itself has bounded average degree. If the bound is a polynomial function, we say that $\mathcal C$ has \textit{polynomial expansion}. These definitions were introduced by Ne\v{s}et\v{r}il and Ossona de Mendez \cite{NOdM08}.

Classes of polynomial expansion have strongly-sublinear separator theorems, and for hereditary classes the two notions are equivalent \cite{DN16}. Thus we may extend Theorem \ref{thm-minor-closed} to classes of polynomial expansion. Such classes include the $1$-planar graphs, that is, the graphs which may be embedded in the plane with each edge crossing at most one other edge once \cite{NOdMW}, and intersection graphs of systems of balls with only a bounded number meeting any point \cite{MTTV}.

Polynomial expansion is a much weaker property than being minor-closed. It is easy to see, for example, that any graph can be suitably subdivided to obtain a $1$-planar graph. However, Borodin \cite{Bor84} showed that all $1$-planar graphs are $6$-colourable, and since they include disjoint unions of $K_6$, we obtain the following tight result from the extension of Theorem \ref{thm-minor-closed} to classes of polynomial expansion.

\begin{cor}The class $\mathcal P_1$ of $1$-planar graphs satisfies $d_{\succ}(\mathcal P_1,t)=(5/3+o_t(1))t$.
\end{cor}

\section{Preliminaries}\label{sec: prelim}

For $n\in \mathbb{N}$, let $[n]:=\{1,\dots, n\}$. Given a set $X$ and $k\in\mathbb{N}$, let $\binom Xk$ the family of all $k$-sets in $X$.  For brevity, we write $v$ for a singleton set $\{v\}$ and $xy$ for a set of pairs $\{x,y\}$.  We write $a=b\pm c$ if $b-c \leq a \leq b+c$. 
If we claim that a result holds whenever we have $0<a\ll b,c\ll d<1$, it means that there exist positive functions $f,g$ such that the result holds as long as $a<f(b,c)$ and $b<g(d)$ and $c<g(d)$. We will not compute these functions explicitly.
In many cases, we treat large numbers as if they are integers, by omitting floors and ceilings if it does not affect the argument. We write $\log$ for the base-$e$ logarithm.

\subsection{Graph notations}
Given graphs $H$ and $G$, in a copy of $H$-subdivision in $G$, we call the vertices that correspond to $V(H)$ the \emph{anchor} vertices of the subdivision.
For a given path $P=x_1\dots x_t$, we write $\Int(P) = \{x_2,\dots, x_{t-1}\}$ to denote the set of its internal vertices. Given a graph $H$, a set of vertices $S\subseteq V(H)$ and a subgraph $F\subseteq H$, denote by $H-S=H[V(H)\setminus S]$ the subgraph induced on $V(H)\setminus S$ and by $H\setminus F$ the spanning subgraph obtained from $H$ by removing edges in $F$.

Given a graph $G$, denote its average degree $2e(G)/\abs{G}$ by $d(G)$. 
For two sets $X,Y \subseteq  V(G)$, the (graph) $\emph{distance}$ between them is the length of a shortest path from $X$ to $Y$.
For two graphs $G,H$, we write $G\cup H$ to denote the graph with vertex set $V(G)\cup V(H)$ and the edge set $E(G)\cup E(H)$. A \emph{$k$-star} denotes a copy of $K_{1,k}$ which is a star with $k$ edges. Given  a collection of graphs $\cF=\{F_i: i\in I\}$, we write $V(\cF)=\bigcup_{i\in I}V(F_i)$ and $\abs{\cF}=\abs{I}$. For path $P$ and a vertex set $U$, we write $P|_U$ for the induced subgraph of $P$ on vertex set $V(P)\cap U$.

For a set of vertices $X\subseteq V(G)$ and $i\in \mathbb{N}$, denote 
\[N^i(X):= \{ u \in V(G): \text{ the distance between }X\text{ and } u\text{ is exactly }i\}\]
the $i$-th sphere around $X$,
and write $N^0(X)=X$, $N(X) := N^1(X)$, and for $i\in\mathbb{N}\cup\{0\}$, let $B^i(X) = \bigcup_{j=0}^{i} N^j(X)$ be the ball of radius $i$ around $X$.
We write $\partial(X)$ for the edge-boundary of $X$, that is, the set of edges between $X$ and $V(G)\setminus X$ in $G$. Given another set $Z\subseteq V(G)$, we write $N(X,Z)=N(X)\cap Z$ for the set of neighbours of $X$ in $Z$.

\subsection{Robust expander}
To define the robust graph expansion, we need the following function. For $\eps_1>0$ and $t>0$, let $\rho(x)$ be
the function
\[\label{epsilon}
\rho(x)=\rho(x,\eps_1,t):=\begin{cases}0 & \text{ if } x<t/5, \\
\eps_1/\log^2(15x/t) & \text{ if } x\ge t/5,\end{cases}\]
where, when it is clear from context we will not write the dependency on $\eps_1$ and $t$ of $\rho(x)$. Note that when $x\ge t/2$, while $\rho(x)$ is decreasing, $\rho(x)\cdot x$ is increasing. 

Koml\'os and Szemer\'edi~\cite{K-Sz94,K-Sz96} introduced a notion of expander $G$ in which any set $X$ of reasonable size expands by a sublinear factor, that is, $\abs{N_G(X)}\ge \rho(\abs{X})\abs{X}$. We shall extend this notion to a robust one such that similar expansion occurs even after removing a relatively small set of edges.

\begin{defin}\label{defn: expander}
	\noindent\textbf{$(\eps_1,t)$-robust-expander:} A graph $G$ is an \emph{$(\eps_1,t)$-robust-expander} if for any subset $X\subseteq V(G)$ of size $t/2\le \abs{X}\le \abs{G}/2$, and any subgraph $F\subseteq G$ with $e(F)\le d(G)\cdot \rho(\abs{X}) \abs{X}$, we have \[\abs{N_{G\setminus F}(X)}\ge \rho(\abs{X}) \abs{X}.\]
\end{defin}

We shall use the following version of expander lemma, stating that every graph contains a robust expander subgraph with almost the same average degree.

\begin{lemma}\label{lem-expander}
	Let $C>30, \eps_1 \leq 1/(10C), \eps_2<1/2, d>0$ and $\rho(x)=\rho(x,\eps_1,\eps_2d)$ as in~\eqref{epsilon}. Then every graph $G$ with $d(G)=d$ has a subgraph $H$ such that
	\begin{itemize}
		\item $H$ is an $\left(\eps_1,\eps_2d\right)$-robust-expander;
		\item $d(H)\geq (1-\delta)d$, where $\delta:=\frac{C\eps_1}{\log 3}$;
		\item $\delta(H)\geq d(H)/2$;
		\item $H$ is $\nu d$-connected, where $\nu:=\frac{\eps_1}{6\log^2(5/\eps_2)}$.
	\end{itemize}
\end{lemma}

We remark that, though almost retaining the average degree, the robust expander subgraph $H$ in Lemma~\ref{lem-expander} could be much smaller than~$G$. For instance, if~$G$ is a union of many vertex disjoint small cliques, then $H$ could be just one of those cliques. Such drawback often makes it difficult to utilise expanders iteratively  within graphs. We include the proof of the robust expander lemma, Lemma~\ref{lem-expander}, in the appendix of the online version.

A key property of the robust expanders that we shall use is that it has small (logarithmic) diameter.

\begin{lemma}[\cite{K-Sz96} Corollary~2.3]\label{lem: path}
	If $G$ is an $n$-vertex $(\eps_1,t)$-robust-expander, then for any two vertex sets $X_1,X_2$ each of size at least $x\geq t/2$, and a vertex set $W$ of size at most $\rho(x)x/4$, there exists a path in $G-W$ between $X_1$ and $X_2$ of length at most 
	\[\frac{2}{\eps_1}\log^3\bfrac{15n}{t}.\]
\end{lemma}

\subsection{Exponential growth for small sets}
In an $(\eps_1,t)$-robust-expander graph, for a set $X$ with size at least $t/2$, the ball $B^i(X)$ grows with the radius $i$. For our purpose, we need to quantify how resilient this growth is to deletion of some thin set around $X$. 

\begin{defin}\label{def: shortest paths}
	For a set $X\subseteq W$ of vertices, the paths $P_1,\ldots,P_q$ are \emph{consecutive shortest paths from $X$ within $W$} if the following holds.
For each $i\in [q]$, $P_i|_W$ is a shortest path from $X$ to some vertex $v_i\in W\setminus X$ in the graph restricted to $W-\bigcup_{j\in [i-1]}V(P_j)$.
\end{defin}

In particular, the following proposition shows that the rate of expansion for small sets is almost exponential in a robust expander even after deleting a few consecutive shortest paths. 

\begin{prop}\label{prop: expansion after deleting shortest path}
Let $0<1/d  \ll  \eps_1, \eps_2 \ll 1$.
Suppose $G$ is an $n$-vertex $(\eps_1,\eps_2 d)$-robust-expander and $X,Y$ are sets of vertices with $\abs{X}= x \geq \eps_2 d$ and  $\abs{Y}\leq \frac{1}{4} \cdot \rho(x) \cdot x$. Let $P_1,\dots, P_q$ be consecutive shortest paths in $G-Y$ from $X$ within $B^{r}_{G-Y}(X)$, where $1\le r\le \log n$ and $q<x \log^{-8}{x}$, and let	$P=\bigcup_{i\in [q]} V(P_i)$.  Then for each $i\in [r]$, we have 
\[\abs{B^{i}_{G-P-Y}(X)} \geq \exp(i^{1/4}).\]
\end{prop} 
\begin{proof}
For each $i\ge 0$, let $Z_i= B^{i}_{G-P-Y}(X)$. 
As $P_1,\dots, P_q$ are consecutive shortest paths from $X$, for each $i\ge 0$, each path $P_j$, $j\in[q]$, can intersect with the set $N_{G-Y}(Z_i)$ on at most $i+2$ vertices. Indeed, otherwise we can replace the initial segment of $P_j$ with a path in $Z_i\cup N_{G-Y}(Z_{i})$ of length $i+1$ to get a shorter path in $G-Y-\bigcup_{k\in [j-1]}V(P_k)$, contradicting the choice of $P_j$. Thus, $\abs{N_{G-Y}(Z_i)\cap P}\le (i+2)q$.
Consequently, the expansion of $G$ implies for each $i\ge 0$ that
	\begin{align*}
		\abs{Z_{i+1}} &= \abs{Z_i} + \abs{N_G(Z_i)\setminus(Y\cup P)}\ge \abs{Z_i} + \rho(\abs{Z_i})\abs{Z_i} - \abs{Y}-\abs{N_{G-Y}(Z_i)\cap P} \\
		&\geq \abs{Z_i} + \frac{3}{4}\rho(\abs{Z_i})\abs{Z_i} - (i+2)q.
	\end{align*}
	
Let
\[f(z)=\exp(z^{1/4}) \quad \text{and} \quad g(z):= x + \frac{1}{2}\rho(x)x z.\]
We first use induction on $i$ to show that for each $0\le i\le \log^4x$, $\abs{Z_i}\ge g(i)$. Indeed, $\abs{Z_0} =\abs{X}=x=g(0)$. Then, as $\rho(z)z$ is increasing when $z\ge x$ and $ \frac{1}{4}\rho(x)x \geq  \frac{(i+2)x}{ \log^8{x}} > (i+2)q$ due to $i\le \log^4x$, we see that 
\[\abs{Z_{i+1}}\ge \abs{Z_i}+ \frac{3}{4}\rho(\abs{Z_i})\abs{Z_i} - (i+2)q\geq \abs{Z_i}+ \frac{1}{2}\rho(x)x = \abs{Z_i}+ g(i+1)-g(i)\ge g(i+1).\]

We may then assume $i>\log^4x$, as $f(i)\le g(i)\le \abs{Z_i}$ when $i\le \log^4x$. Now, as $i>\log^4x$, $\frac{f(i)}{i^{7/4}}\ge \frac{f(\log^4x)}{(\log^4x)^{7/4}}=\frac{x}{\log^7x}$ and so
\[ (i+2)q<i\cdot \frac{2x}{\log^8x}\le i\cdot \frac{f(i)}{i^{7/4}}= \frac{ f(i)}{i^{3/4}}.\]
Also note that $f(i+1)-f(i)\le \frac{f(i)}{i^{3/4}}$ and $\rho(\abs{Z_i})\abs{Z_i}\ge \rho(f(i))f(i)\ge\frac{\eps_1f(i)}{i^{1/2}}$.
Thus, we have
\begin{align*}
\abs{Z_{i+1}}&\ge \abs{Z_i}+\frac{3}{4}\rho(\abs{Z_i})\abs{Z_i} - (i+2)q \geq 
\abs{Z_i}+\frac{3\eps_1f(i)}{4i^{1/2}}- \frac{ f(i)}{i^{3/4}}\\
&\geq \abs{Z_i}+\frac{ f(i)}{i^{3/4}} \geq \abs{Z_i}+ f(i+1)-f(i)\ge f(i+1),
\end{align*}
as desired.
\end{proof}

\subsection{Basic building structures}
The following structures will serve as basic building blocks for our constructions of subdivisions.

\begin{defin}[$(2a,b)$-Sun]\label{defn-sun}
	For integers $a\geq b\geq 0$, a \textit{$(2a,b)$-sun} is a bipartite graph consisting of a cycle $x_1,\dots, x_{2a}$  and leaves $y_1,\dots, y_b$, where for each $i\in [a]$ the vertex $x_{2i}$ is adjacent to at most one leaf, and for each $i\in [b]$ the leaf $y_i$ is adjacent to a vertex $x_{2j}$ for some $j\in [a]$.
\end{defin}
Note that a $(2a,0)$-sun is just an even cycle of length $2a$.
\begin{figure}[ht]
\begin{minipage}[t]{.3\textwidth}
\begin{center}
\begin{tikzpicture}[thick,scale=.85]
\filldraw (0:2) circle (0.05);
\filldraw (36:2) circle (0.05);
\filldraw (72:2) circle (0.05);
\filldraw (108:2) circle (0.05);
\filldraw (144:2) circle (0.05) node [anchor=north west] {$x_{2a}$};
\filldraw (180:2) circle (0.05) node [anchor=west] {$x_1$};
\filldraw (216:2) circle (0.05) node [anchor=south west] {$x_2$};
\filldraw (252:2) circle (0.05) node [anchor=south west] {$\cdots$};
\filldraw (288:2) circle (0.05);
\filldraw (324:2) circle (0.05);
\draw (0:2) -- (36:2) -- (72:2) -- (108:2) -- (144:2) -- (180:2) -- (216:2) -- (252:2) -- (288:2) -- (324:2) -- cycle;
\filldraw (0:3) circle (0.05);
\draw (0:2) -- (0:3);
\filldraw (72:3) circle (0.05);
\draw (72:2) -- (72:3);
\filldraw (216:3) circle (0.05);
\draw (216:2) -- (216:3);
\draw[white] (0,3.57061) circle (1); 
\end{tikzpicture}
\end{center}
\caption{A $(2a,b)$-sun with $a=5$ and $b=3$.}
\end{minipage}\hfill\begin{minipage}[t]{.65\textwidth}
\begin{center}
\begin{tikzpicture}[thick, scale=.75]
\filldraw[fill=purple!35!white] (0,0) circle (.7);
\draw[purple,<-] (-45:.8) -- (-45:1.2) node [anchor=west] {head};
\node at (-90:.8) [anchor=north] {M};

\draw (0:5) circle (.7);
\path (0:5) +(90:.8) node [anchor=south west] {$\abs{D_4}\leq s$};
\path (0:5) +(-45:.8) coordinate (leg1);
\path (0:5) +(-45:1.2) coordinate (leg2);
\draw[blue,<-] (leg1) -- (leg2) node [anchor=west] {leg};
\draw (60:5) circle (.7);
\path (60:5) +(90:.8) node [anchor=south] {$D_3$};
\draw (120:5) circle (.7);
\path (120:5) +(90:.8) node [anchor=south] {$D_2$};
\draw[purple, ->] (120:5) -- ++(90:.6) node[pos=.5, anchor=west] {$r$};
\draw (180:5) circle (.7);
\path (180:5) +(90:.8) node [anchor=south east] {$D_1$};

\path (180:5) +(60:1) coordinate (a1);
\path (120:5) +(-120:1) coordinate (a2);
\path (120:5) +(0:1) coordinate (a3);
\path (60:5) +(180:1) coordinate (a4);
\path (60:5) +(-60:1) coordinate (a5);
\path (0:5) +(120:1) coordinate (a6);

\draw[loosely dashed,blue] (a1) -- (a2);
\draw[loosely dashed,blue] (a3) -- (a4);
\draw[loosely dashed,blue] (a5) -- (a6) node [pos=0.5, anchor=south west] {distance $\geq\tau$};
\draw[loosely dashed,blue] (4,-.35) -- (1,-.35);

\draw plot [smooth] coordinates {(180:.7) (170:2) (-175:3.5) (180:4.3)};
\node at (178:2.8) [anchor=south] {$P_1$};
\draw[purple,<->] (-.7,-.7) -- (-4.3,-.7) node [pos=.5, fill=white] {$\leq 10r$};
\draw plot [smooth] coordinates {(0:.7) (0:1.5) (15:3) (0:4.3)};
\node at (3,0) [anchor=south] {$P_4$};
\draw plot [smooth] coordinates {(60:.7) (70:2.2) (60:3.5) (60:4.3)};
\node at (67:3) [anchor=south] {$P_3$};
\draw plot [smooth] coordinates {(120:.7) (100:2) (130:3) (120:4.3)};
\node at (120:2.7) [anchor=south] {$P_2$};
\end{tikzpicture}
\end{center}
\caption{A $(4,s,r,\tau)$-nakji (Definition \ref{defn-nakji}).}
\end{minipage}
\end{figure}

\begin{defin}[$(h_1,h_2,h_3)$-unit]\label{defn-unit}
	For $h_1,h_2,h_3 \in \mathbb{N}$, a graph $F$ is an \emph{$(h_1,h_2,h_3)$-unit}~if it contains distinct vertices $u$ (the \emph{core vertex} of $F$) and $x_1,\ldots,x_{h_1}$, and $F = \bigcup_{i \in [h_1]}(P_i \cup S_{x_i})$, where
	\begin{itemize}
		\item $\lbrace P_i: i\in[h_1]\rbrace$ is a collection of pairwise internally vertex disjoint paths, each of length at most $h_3$, such that $P_i$ is a $u,x_i$-path.
		\item $\lbrace S_{x_i}:i\in [h_1]\rbrace$ is a collection of vertex disjoint $h_2$-stars such that $S_{x_i}$ has centre $x_i$ and $\bigcup_{i \in [h_1]} (V(S_{x_i})\setminus \lbrace x_i\rbrace)$ is disjoint from $\bigcup_{i \in [h_1]} V(P_i)$.
	\end{itemize}
	Define the \emph{exterior} $\Ext(F):=\bigcup_{i \in [h_1]}(V(S_{x_i})\setminus  \lbrace x_i\rbrace )$ and \emph{interior} $\Int(F):=V(F)\setminus \Ext(F)$.
	For every vertex $w\in \Ext(F)$, let $P(F,w)$ be the unique path from the core vertex $u$ to $w$ in $F$.
\end{defin}

\begin{defin}[$(h_0,h_1,h_2,h_3)$-web]\label{defn-web}
	For $h_0,h_1,h_2,h_3 \in \mathbb{N}$, a graph $W$ is an \emph{$(h_0,h_1,h_2,h_3)$-web}~if it contains distinct vertices $v$ (the \emph{core vertex} of $W$) and $u_1,\ldots,u_{h_0}$, and  $W = \bigcup_{i \in [h_0]}(Q_i \cup F_{u_i})$, where
	\begin{itemize}
		\item $\lbrace Q_i:i\in[h_0]\rbrace$ is a collection of pairwise internally vertex disjoint paths such that $Q_i$ is a $v,u_i$-path of length at most $h_3$.
		\item $\lbrace F_{u_i}:i\in [h_0]\rbrace$ is a collection of vertex disjoint $(h_1,h_2,h_3)$-units such that $F_{u_i}$ has core vertex $u_i$ and $\bigcup_{i \in [h_0]} (V(F_{u_i})\setminus\lbrace u_i\rbrace)$ is disjoint from $\bigcup_{i \in [h_0]} V(Q_i)$.
	\end{itemize}
	Define the \emph{exterior} $\Ext(W):=\bigcup_{i \in [h_0]}\Ext(F_{u_i})$, \emph{interior} $\Int(W):=V(W)\setminus \Ext(W)$ and \emph{centre} $\Cen(W):= \bigcup_{i \in [h_0]}V(Q_i)$.
	For every vertex $w \in \Ext(W)$, let $P(W,w)$ be the unique path from the core vertex $v$ to $w$ in $W$.
\end{defin}

\begin{figure}[ht]
\begin{center}
\begin{tikzpicture}[thick, scale=.85]
\filldraw[orange!25!white, rounded corners] (0.25,3.5) rectangle (4.6,-3);
\draw[orange,->] (2,-3.5) node[anchor=north]{centre} -- (2,-3.1);
\draw [ultra thick, green!80!black, decoration={brace,mirror,amplitude=5pt}, decorate] (-.15,-5) -- (8.15,-5) node [pos=0.5,anchor=north,yshift=-5pt] {interior};
\draw[purple, rounded corners] (4.25,-1) rectangle (9,-4.5);
\node at (6.75,-1) [purple, anchor=south] {$(h_1,h_2,h_3)$-unit}; 

\draw[blue] (0,0) circle (0.15);
\filldraw (0,0) circle (0.05);
\node at (-.15,-.15) [anchor=south east,blue] {core vertex};
\filldraw (4.5,3) circle (0.05);
\draw plot [smooth] coordinates {(0,0) (1,1.25) (3,2) (4.5,3)};
\filldraw (4.5,1.5) circle (0.05);
\draw plot [smooth] coordinates {(0,0) (1,.4) (2,.5) (3.5,1) (4.5,1.5)};
\node at (5,1.5) [yshift=1ex] {$\vdots$};
\node at (3.5,0) {$\vdots$};
\filldraw (4.5,-2.5) circle (0.05);
\draw plot [smooth] coordinates {(0,0) (1,-.5) (2.5,-1.5) (3.5,-2.35) (4.5,-2.5)};
\draw[<->, blue!50!purple] (0,0.5) .. controls (0,2) and (3,3.25) .. (4.5,3.25) node[midway, fill=orange!25!white] {$\leq h_3$};
\draw[purple] (-45:0.5) arc [start angle=-45, end angle=75, radius=0.5];
\node [purple, anchor=west] at (0.5,0) {$h_0$ branches};

\filldraw (8,4.5) circle (0.05);
\draw (8,4.5) -- ++(.75,.25) -- ++(0,-.5) -- cycle;
\draw plot [smooth] coordinates {(4.5,3) (6.25,3.75) (7,4) (7.5,4.6) (8,4.5)};

\filldraw (8,3.5) circle (0.05);
\draw (8,3.5) -- ++(.75,.25) -- ++(0,-.5) -- cycle;
\draw plot [smooth] coordinates {(4.5,3) (5.5,3.1) (7,3.5) (8,3.5)};

\filldraw (8,2.5) circle (0.05);
\draw (8,2.5) -- ++(.75,.25) -- ++(0,-.5) -- cycle;
\draw[green!80!black] (8,2.5) ++(-45:0.5) node [anchor=north] {$h_2$-star} arc [start angle=-45, end angle=45, radius=0.5];
\draw plot [smooth] coordinates {(4.5,3) (5.5,2.75) (6,2.6) (7,2.7) (8,2.5)};

\draw[blue] (4.5,3) ++(-45:0.5) arc [start angle=-45, end angle=60, radius=0.5] node [anchor=south] {$h_1$ branches};

\filldraw (8,-2) circle (0.05);
\draw (8,-2) -- ++(.75,.25) -- ++(0,-.5) -- cycle;
\draw plot [smooth] coordinates {(4.5,-2.5) (5.5,-2.25) (5.75,-2.4) (7,-2.1) (8,-2)};

\filldraw (8,-3) circle (0.05);
\draw (8,-3) -- ++(.75,.25) -- ++(0,-.5) -- cycle;
\draw plot [smooth] coordinates {(4.5,-2.5) (6,-2.9) (7,-2.9) (7.5,-2.8) (8,-3)};

\filldraw (8,-4) circle (0.05);
\draw (8,-4) -- ++(.75,.25) -- ++(0,-.5) -- cycle;
\draw plot [smooth] coordinates {(4.5,-2.5) (5,-3) (6,-3.4) (7.5,-3.5) (8,-4)};

\draw[<->, blue!50!purple] (4.5,-1.5) -- (8,-1.5) node[midway, fill=white] {$\leq h_3$};

\end{tikzpicture}
\end{center}
\caption{An $(h_0,h_1,h_2,h_3)$-web.}
\end{figure}
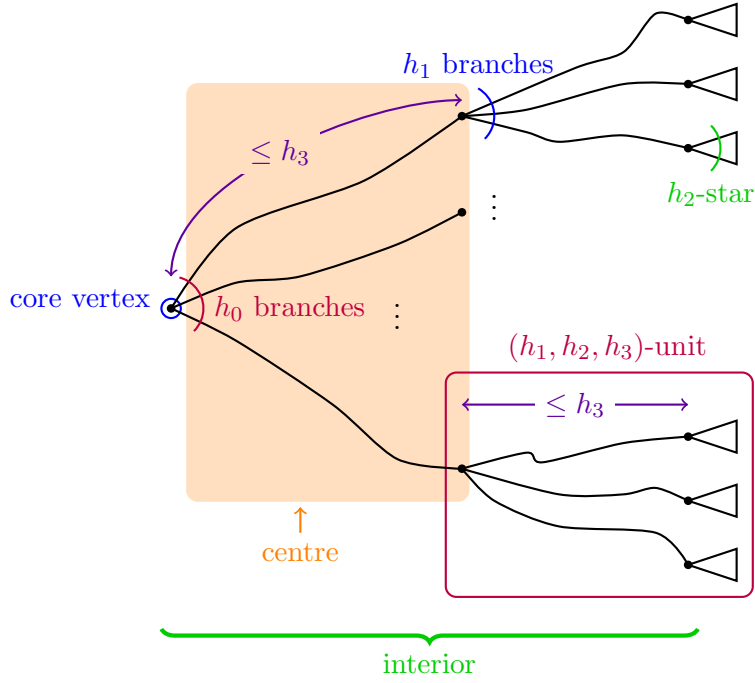

\begin{defin}[$(t,s,r,\tau)$-nakji]\label{defn-nakji}
	Given $t,s,r,\tau\in\mathbb{N}$, a graph $N$ is a \emph{$(t,s,r,\tau)$-nakji} \footnote{Nakji means `long arm octopus' in Korean.} in $G$ if it contains vertex disjoint sets $M$, $D_i$, $i\in [t]$, each having size at most $s$, and paths $P_i$, $i\in[t]$  such that for each $i\in[t]$
	\begin{itemize}
		\item  $P_i$ is an $M,D_i$-path with length at most $10r$, and all paths $P_i$, $i\in[t]$, are pairwise internally disjoint;\label{S1}
		
		\item $D_i$ has diameter at most $r$, and all $D_i$, $i\in[t]$, are a distance at least $\tau$ in $G$ from each other and from $M$, and they are disjoint from internal vertices of $\bigcup_{i\in[t]}P_i$.\label{S2}
	\end{itemize}
	We call $M$ the \emph{head} of the nakji and each $D_i$, $i\in[t]$, a \emph{leg}.
\end{defin}

\section{Outline of the proofs}\label{sec: outline}
To prove Theorem~\ref{thm: main}, we first use Lemma~\ref{lem-expander} to pass to a robust expander subgraph without  losing much on the average  degree. Depending on the density of the expander, we use different approaches. Roughly speaking, when the expander has positive edge density, we will utilise pseudorandomness via the machinery of the graph regularity lemma and the blow-up lemma (Lemma~\ref{lem: dense});  while if the expander is not dense, then we exploit its sublinear expansion property (Lemmas~\ref{lem: moderate} and~\ref{lem: very sparse}).

\begin{lemma}\label{lem: dense}
	Let $0< 1/d, \alpha \ll \delta, \eta, 1/\Delta <1$ and $\delta\ll\eps$ and $H$ be a graph with at most $(1-\eps)d$ vertices with $\Delta(H)\leq \Delta$ and $G$ is an $n$-vertex graph.	Suppose that $d> \eta n$ and $H$ is bipartite and $\alpha$-separable.
	If $d(G)\geq (1-\delta)d$, then $G$ contains $H$ as a subgraph.
\end{lemma}

\begin{lemma}\label{lem: moderate}
Let $0<1/d\ll \eta\ll\eps_1,\eps_2\ll \eps, 1/\Delta<1$, and $H$ be a graph with at most $(1-\eps)d$ vertices with $\Delta(H)\leq \Delta$ and $G$ is an $n$-vertex graph and $m=\frac{2}{\eps_1}\log^3\bfrac{15n}{\eps_2d}$. Suppose $G$ is an $(\eps_1,\eps_2d)$-robust-expander with $d(G)\geq \eps^2 d$.  If $m^{100}\leq d\leq \eta n$, then $G$ contains an $H$-subdivision.
\end{lemma}

\begin{lemma}\label{lem: very sparse}
Let $0<1/d\ll \nu,\delta \ll\eps_1,\eps_2\ll \eps, 1/\Delta<1$,  and $H$ be a graph with at most $(1-\eps)d$ vertices with $\Delta(H)\leq \Delta$ and $G$ is an $n$-vertex graph and $m=\frac{2}{\eps_1}\log^3\bfrac{15n}{\eps_2d}$. Suppose $H$ is bipartite and $G$ is an $(\eps_1,\eps_2d)$-robust-expander with $d(G)\geq (1-\delta)d$ and $\delta(G)\geq d(G)/2$. If $d < m^{100}$, then $G$ contains an $H$-subdivision.
\end{lemma}

Our main theorem readily follows, assuming these three lemmas. 

\begin{proof}[Proof of Theorem~\ref{thm: main}]
For given $\eps$, let $\eta, \nu,\delta, \eps_1,\eps_2, \alpha_0, 1/d_0$ be small enough so that Lemma~\ref{lem: dense}--\ref{lem: very sparse} holds for all $\alpha \leq \alpha_0$ and $d>d_0$.
We apply Lemma~\ref{lem-expander} to $G$ to obtain a subgraph $G'$ which is an $n'$-vertex $(\eps_1,\eps_2d)$-robust-expander and $d(G')\geq (1-\delta)d$.

If $d\ge \eta n'$, then as $H$ is $\alpha$-separable, Lemma~\ref{lem: dense} implies that $H\subseteq G'$.
If $d < \eta n'$, then Lemmas~\ref{lem: moderate} and~\ref{lem: very sparse} imply that $G$ contains an $H$-subdivision. This proves the theorem.	
\end{proof}

In the rest of this section, we outline the ideas in our constructions. Let $G$ be an expander and $H$ be the bounded-degree bipartite graph whose subdivision we want to embed in $G$.

\subsection{Embeddings in dense graphs}
The regularity lemma essentially partitions our graph $G$ into a bounded number of parts, in which the bipartite subgraphs induced by most of the pairs of parts behave pseudorandomly. The information of this partition is then stored to a (weighted) fixed-size so-called reduced graph $R$ which inherits the density of $G$. 

We seek to embed $H$ in $G$ using the blow-up lemma, which boils down to finding a `balanced' bounded-degree homomorphic image of $H$ in $R$. This is where the additional separable assumption on $H$ kicks in. The separability of $H$ (in the form of sublinear bandwidth) enables us to cut $H$ into small pieces to offer suitable `balanced' homomorphic images.

Now, if the reduced graph $R$ has chromatic number at least three, the density  of $R$ inherited from $G$ is just large enough to guarantee an odd cycle long enough in $R$ to serve as our bounded-degree homomorphic image of $H$. 

It is, however, possible that $R$ is a bipartite graph, hence all cycles within are even cycles, which might not be long enough for our purpose. Indeed, in the worst case scenario that $H$ is an extremely asymmetric bipartite graph, an even cycle has to be twice as long as that of an odd one to be useful. This is because in the odd cycle we can circumvent the asymmetry of $H$ by `breaking the parity' via wrapping around the odd cycle twice.

To handle the case when $R$ is bipartite, instead of cycles, we will make use of the sun structure (Definition~\ref{defn-sun}), in which the leaves attaching to the main body of sun help in balancing out the asymmetry of $H$.

\subsection{Embeddings in robust expanders with medium density}
The robust expansion underpins all of our constructions of $H$-subdivisions when the graph $G$ is no longer dense. At a high level, in $G$, we anchor on some carefully chosen vertices and embed paths between anchors (corresponding to the edge set of $H$) one at a time. 

As these paths in the subdivision need to be internally vertex disjoint, to realise this greedy approach we will need to build a path avoiding a certain set of vertices. This set of vertices to avoid contains previous paths that we have already found and often some small set of `fragile' vertices that we wish to keep free.

To carry out such robust connections, we use the small-diameter property of expanders (Lemma~\ref{lem: path}). Let $m$ be the diameter of $G$. Recall that $H$ is of order at most $d$ with bounded degree and we need to embed $e(H)=O(d)$ paths. Thus, all in all, the set of vertices involved in all connections, say $W$, is of size $O(dm)$. To enjoy Lemma~\ref{lem: path}, we want to anchor at vertices with large `boundary' compared to $W$, that is, being able to access many ($dm^{10}$ say) vertices within short distance. 

With that being said, if there are now $d$ vertices of high degree (at least $dm^{10}$), we can easily finish the embedding anchoring on these high degree vertices. 
This almost enables us to view $G$ as if it is a `relatively regular' graph. In reality, what happens is that we can assume the set of high degree vertices, $L$, is small: $\abs{L}<d$; and deduce from the robust expansion property that, $G-L$ is still dense. It is worth pointing out that this is where we need to extend the original notion of expander to this robust one.

Without high degree vertices, we turn to the web structure (Definition~\ref{defn-web}), in which the core vertex has a large `boundary' (the web's exterior) of size about $dm^9$. If there are $d$ webs of suitable size, we can then anchor on their core vertices and connect pairs via the exteriors of the corresponding webs. One thing to be careful about is that a web will become useless if the (few) vertices in its centre are involved in previous connections. To prevent this, when constructing paths between exteriors, we protect the `fragile' centres of all webs.

Lastly, as $G$ is `relatively regular', we can pull out many large stars (size roughly $d$) and link them up to find the webs one by one. The construction of webs is one of the places we require $G$ to be not too sparse ($d$ being a large power of $m$ suffices).

\subsection{Embeddings in sparse robust expanders}
The current way of building and connecting webs breaks down if the expander is too sparse, say with average degree at most $\log\log n$. We will have to rely on other structures to build subdivisions in this case.

Let us first look at the easier problem of finding $H$-minors, in which case we just need to find $d$ large balls (and contract them afterwards) and find $O(d)$ internally disjoint paths between them. Note that now $\abs{L}<d\le m^{100}$ is quite small. Suppose additionally that $G-L$ has average degree $\Omega(d)$ and within it we can  find $d$ vertices, $v_1,\ldots, v_d$, pairwise a distance $\sqrt{\log n}$ apart, such that for each $v_i$, the ball $B_i$ of radius say $(\log\log n)^{20}$ around it has size at least $m^{200}$. So each $B_i$ is large enough to enjoy exponential growth (Proposition~\ref{prop: expansion after deleting shortest path}) avoiding all paths previously built. Now to get, say, a $v_i,v_j$-path, we first expand $B_i, B_j$ to larger balls with radius say $(\log n)^{1/10}$. These larger balls are so gigantic that we can connect them avoiding all the smaller balls $\bigcup_{i\in[d]}B_i$. It is left to find such $v_i$ and $B_i$. We can find them one by one, by collectively growing a set $U$ of pairwise far apart vertices past $L$ and using an averaging argument to locate the next $v_i$ that expands well in $G-L$. 

Coming back to embedding $H$-subdivisions, we shall follow the general strategy as that of finding minors. However, an immediate obstacle we encounter is the following. To get a subdivision instead of a minor, we need to be able to lead up to $\Delta(H)=O(1)$ many paths arriving at $B_i$ disjointly to $v_i$. In other words, each anchor vertex $v_i$ has to expand even after removing $O(1)$ disjoint paths starting from itself. Here comes the problem: in the minor case, we just need to expand $U$ ignoring a smaller set $L$; whereas now $U$ is asked to expand past a larger set of $\Theta(\abs{U})$ vertices that are used in the previous connections. Our expansion property is simply too weak for this.

This is where the nakji structure (Definition~\ref{defn-nakji}) comes into play. It is designed precisely to circumvent this problem by doing everything in reverse order. Basically, instead of looking for anchor vertices that expand robustly, we rather anchor on nakjis and link them via their legs first and then extend the paths from the legs in each nakji's head using connectivity.

Why do we require that the head and legs of a nakji are stretched far apart? This is so that, before linking nakjis together to a subdivision, we can expand each leg without bumping into any other part to an enormous size, so that each connection made leaves irrelevant structures untouched. 

The remaining task is then to find many nakjis in $G-L$. This is done essentially by linking small subexpanders within $G-L$. A subtlety here worth pointing out is that we only have $\abs{L}<d$, while each of the subexpanders, though having size $\Omega(d)$,  could be smaller than $L$. This keeps us from expanding and linking each subexpander in $G-L$. Intuitively, given that $L$ is not large, one would like to take a huge set of subexpanders, whose union is so large and thus grows easily past $L$. However, as the expanding function $\rho(\cdot)$ is sublinear, if there are too many subexpanders to begin with, after averaging, the expansion rate of each subexpander in $G-L$ is too weak to be useful. To overcome this difficulty, instead, we shall average over a set of subexpanders of appropriate size that is just big enough to ignore $L$ and on the other hand just small enough that $\rho(\cdot)$ does not decay too much.

\section{Separable bipartite graphs in dense graphs}\label{sec: dense}

In this section, we prove Lemma~\ref{lem: dense}. We will use Szemer\'edi's regularity lemma; for a detailed survey of this lemma and its numerous applications, see~\cite{ksssurvey,kssurvey}.
Let $d_G(A,B)=\frac{e(G[A,B])}{\abs A \abs B}$ be 
the \emph{density} of a bipartite graph $G$ with vertex classes $A$ and $B$. 
For a positive number $\eps > 0$, we say that a bipartite graph $G$ with vertex classes $A$ and $B$ is $\eps$\emph{-regular} if every $X \subseteq A$ and $Y \subseteq B$ with $\abs X\geq \eps \abs A$ 
and $\abs Y \geq \eps \abs B$ satisfy $\abs{d_G(A,B) - d_G(X,Y)} \leq \eps$.  
We say that it is $(\eps,\delta^+)$-regular if it is $\eps$-regular and $d_G(A,B)\geq \delta$.

The following is a version of the regularity lemma suitable for our purpose.

\begin{lemma} \label{lem: reg}
	Suppose $0< 1/n \ll  1/r_1 \ll \eps \ll \tau <1$ with $n\in \mathbb{N}$ and let $G$ be an $n$-vertex graph with $d(G)\geq d n$ for some $d\in(0,1)$. 
	Then there is a partition of the vertex set of $G$ into $V_1, \ldots , V_{r}$ and a graph $R$ on the vertex set $[r]$ such that the following holds:
	\begin{enumerate}[label=(\roman*)]
		\item $1/\eps \leq r \leq r_1$;
		\item for each $i\in [r]$, we have $\abs{V_i} = \frac{(1\pm \eps)n}{r}$;
		\item $d(R)\geq (d-2\tau)r$;
		\item for all $ij\in E(R)$, the graph $G[V_i,V_j]$ is $(\eps,\tau^+)$-regular.
	\end{enumerate}
\end{lemma}

The graph $R$ above is often called the \emph{reduced graph} of $G$ (with respect to the partition $V_1,\ldots, V_r$).
One of the reasons why the regularity lemma is useful is that it can be combined with the blow-up lemma of Koml\'os, S\'ark\"ozy and Szemer\'edi (see \cite[Remark 8]{KSSblowup}).
Here, we only need the following weaker version of the blow-up lemma which only yields a non-spanning subgraph in a given graph.
\begin{theorem}\label{blowup}\label{thm: blow up}
	Suppose $0< 1/n \ll \eps \ll \tau, 1/\Delta <1$ and $1/n \ll 1/r$ with $r, n\in \mathbb{N}$. Suppose $H$ is a graph with $\Delta(H)\leq \Delta$ having a vertex partition $X_1\cup \dots \cup X_r$, $G$ is a graph with the vertex partition $V_1\cup \dots \cup V_r$, and $R$ is a graph on the vertex set $[r]$ with $\Delta(R)\leq \Delta$.
	Suppose further that the followings hold.
	\begin{enumerate}[label=(\roman*)]
		\item\label{blow1} For each $i\in [r]$, $X_i$ is an independent set in $H$ and $\abs{X_i}\leq (1-\eps^{1/2}) \abs{V_i}$.
		\item For each $ij\in E(R)$, the graph $G[V_i,V_j]$ is $(\eps,\tau^+)$-regular.
		\item\label{blow3} For each $\{i,j\}\in \binom{[r]}{2}$ with $ij\notin E(R)$, the graph $H$ contains no edges from $X_i$ to $X_j$. 
	\end{enumerate}
	Then $G$ contains $H$ as a subgraph.
\end{theorem}

\subsection{Balanced homomorphic image of \texorpdfstring{$H$}{H}}
We now show how to find a suitable partition of $H$ to invoke the blow-up lemma. This will use the separability property; however, we find it more convenient to work with bandwidth.
A graph $H$ has \emph{bandwidth} $b$ if we can order its vertices $x_1,\dots, x_{\abs{H}}$ of $H$ such that $x_ix_j\notin E(H)$ for $\abs{i-j}>b$.
In general, small bandwidth is a stronger notion than small separability. However the following result of B\"ottcher, Pruessmann, Taraz and W\"urfl shows that for bounded-degree graphs the two notions are roughly equivalent.

\begin{lemma}\cite[Theorem 5]{BPTW}
	Suppose $0< 1/d \ll \alpha \ll \beta \ll 1/ \Delta\leq 1$. If $H$ is an $\alpha$-separable graph with at most $d$ vertices and $\Delta(H)\leq \Delta$, then $H$ has bandwidth at most $\beta d$.
\end{lemma}

We want to partition $H$ into almost equal sized sets $X_1\cup\dots\cup X_r$ for an appropriate $r$ where all edges of $H$ lie between two consecutive sets $X_{i}$ and $X_{i+1}$ where $X_{r+1}=X_1$. In other words, we want to find a `balanced' homomorphism from $H$ into $C_r$.

Later we will apply regularity lemma to $G$ and find a cycle $C$ in the reduced graph we obtain from the regularity lemma, and apply the blow-up lemma with the above partition of $H$ where $r=\abs{C}$. If $\abs{C}$ is odd, then the value $r$ is large enough for us to fit $H$ into $G[\bigcup_{i\in C} V_i]$. However, if $\abs{C}$ is even and $H$ is an unbalanced bipartite graph then our strategy does not work. For such a case, we need to consider a sun instead.

The following two lemmas provide partitions of $H$ suitable for our purpose. 
First one finds a `balanced' homomorphism of $H$ into an odd cycle $C_r$ and the second finds a `balanced' homomorphism of $H$ into a suitable sun. As two proofs are similar, we omit the proof of the first lemma here. It will be available in the appendix of the online version of this paper.

\begin{lemma}\label{lem: H partition odd}
	Suppose $0< 1/d \ll \beta \ll 1/r, \delta \ll  1/ \Delta <1$ and $r$ is an odd integer. If $H$ is a bipartite graph with at most $(1-\delta)d$ vertices and $\Delta(H)\leq \Delta$ and bandwidth at most $\beta d$, then we can find a partition $X_1\cup \dots \cup X_{r}$ such that the followings hold.
	\begin{itemize}
		\item For each $i\in [r]$, we have $\abs{X_i} \leq \frac{d}{r}$.
		\item Every edge of $H$ is between $X_{i}$ and $X_{i+1}$ for some $i\in [r]$ where $X_{r+1}=X_1$.
	\end{itemize}
\end{lemma}

\begin{lemma}\label{lem: H partition even}
	Suppose $1/d \ll \beta \ll 1/r, \delta \ll  \eps, 1/ \Delta \leq 1$ and $R_0$ is an $(2s,q)$-sun with $s+q \geq r$ and $s\geq q$.
	Suppose that $H$ is a bipartite graph with bipartition $A\cup B$ having at most $(1-\delta)d$ vertices and $\Delta(G)\leq \Delta$.
	Then we can find a partition $\{ X_{u} : u\in V(R_0)\}$ of $V(H)$  as follows.
	\begin{itemize}
		\item For each $u\in V(R_0)$, we have $\abs{X_u}\leq \frac{d}{r}$
		\item Every edge of $H$ is between $X_u$ and   $X_v$ for some $uv\in E(R_0)$.
	\end{itemize}
\end{lemma}
\begin{proof}
	As $H$ has bandwidth at most $\beta d$, there exists an ordering $x_1,\dots, x_{\abs{H}}$ of $V(H)$ such that $x_ix_j\in E(H)$ implies $\abs{i-j}\leq \beta d$.
	Let $A\cup B$ be a bipartition of $H$ such that $\abs{B}= \gamma \abs{A}\leq d$ for some $\frac{1}{\Delta+1}\leq \gamma\leq  1$. Let $u_1\dots u_{2s}$ be the cycle in the sun $R_0$ and assume $u_{p_1},\dots, u_{p_{q}}$ be the vertices which has a leaf-neighbour  such that $p_1,\dots, p_q$ are all even. We consider all indices in sun (respectively in partition of $H$) up to modulo $2s$ (respectively $r$), i.e.~$u_{a+2s}=u_{a}$ and $X_{a+r}=X_r$ for all $a\in \mathbb{N}$. Let $v_{p_i}$ be the leaf neighbour of $u_{p_i}$ for each $i\in [q]$.

	Let $t= \lceil \beta^{-1/2}\rceil$  and we divide the vertices of $H$ according to the ordering as follows: for each $i \in [t^2]$ and $C\in \{A,B\}$, let 
	\[Y^C_i = \{ x_j \in C : (i-1)\beta d < j \leq i\beta d\} \quad \text{and} \quad  Y_i=Y_i^A\cup Y_i^B.\]
	Note that this guarantees that no edge of $H$ is between $Y_i$ and $Y_{j}$ with $\abs{i-j}>1$. For each $i\in [t]$ and $C\in \{A,B\}$, let 
	\[W_i^C = \bigcup_{j\in [r]} Y_{(i-1)t+j}^C \quad \text{and}  \quad Z_i^C = \bigcup_{j \in [t]\setminus[r]} Y_{(i-1)t+j}^C.\]
	
	In the claim below, we will decide to which part $X_\ell$ we assign the vertices in $Z_i^C$. To make such an assignment possible while keeping the edges only between two consecutive parts, we allow the vertices in $W_i^C$ to be assigned to some other parts. As each set $W_i^C$ is much smaller than $Z_i^C$, the uncontrolled assignments of $W_i^C$	will not harm us too much. 
Indeed, the set  $W=\bigcup_{i\in [t]}(W_i^A\cup W_i^B)$ has size at most $r \beta^{1/2} d < \delta^2 d/r$.

Now, we will decide to which part $X_u$ we will assign $Z_i^C$.
For this, we partition the set $[t]$ into $I_1,\dots, I_s, J_1,\dots, J_q$ as in the following claim. If $i\in I_{\ell}$, then we will later assign the vertices in $Z_i^A$ to $X_{u_{2\ell-1}}$ and the vertices in $Z_{i}^B$ to $X_{u_{2\ell}}$. 
If $i\in J_{\ell'}$, then we will later assign the vertices in $Z_i^A$ to $X_{v_{p_{\ell'}}}$ and 
$Z_i^B$ to $X_{u_{p_{\ell'}}}$. As we do not know how unbalanced two sets $Z_i^A$ and $Z_i^B$ are for each $i$, we prove the following claim using random assignments.

	\begin{claim}
		There exists a partition $I_1,\dots, I_s, J_1,\dots, J_q$ of $[t]$ satisfying the following.
		\begin{itemize}
			\item For each $\ell \in [s]$ and $C\in \{A,B\}$, we have $\abs{\bigcup_{i\in I_{\ell}} Z_i^C} \leq \frac{d}{r} $.
			\item For each $\ell \in[s], \ell'\in [q]$, we have 
			$\bigabs{\bigcup_{i\in J_{\ell}} Z_i^A} \leq \frac{d}{r}$ and $\bigabs{\bigcup_{i\in J_{\ell'}} Z_i^B\cup \bigcup_{i\in I_{p_{\ell'}}} Z_i^B}\leq \frac{d}{r}$.
		\end{itemize}	
	\end{claim}
	\begin{poc}
		We add each $\ell\in [t]$ independently to one of $I_1,\dots, I_s, J_1,\dots, J_q$ uniformly at random. Note that for each set, $\ell$ is in the set with probability $1/(s+q) \leq 1/r$.
		Standard concentration inequalities (e.g.\ Azuma's inequality) easily show that with positive probability, we can ensure that for each 
		$\ell \in [s]$ and $\ell'\in [q]$ we have
		\begin{align*}
		\Bigabs{\bigcup_{i\in I_{\ell}} Z_i^A}, \Bigabs{\bigcup_{i\in I_{\ell}} Z_i^B}, \Bigabs{\bigcup_{i\in J_{\ell'}} Z_i^A} &\leq \frac{\abs{A}}{r} + \beta^{1/5}d \leq \frac{(1-\delta)d}{(1+\gamma)r}+ \beta^{1/5}d \leq  \frac{(1-\delta^2)d}{r},	\\
		\Bigabs{\bigcup_{i\in I_{p_{\ell'}}} Z_i^B \cup \bigcup_{i\in J_{\ell'}} Z_i^B} &\leq \frac{2\abs{B}}{r} + \beta^{1/5}d \leq \frac{2\gamma (1-\delta) d}{(1+\gamma)r} + \beta^{1/5}d  \leq  \frac{(1-\delta^2) d}{r}.
		\end{align*}
		Here, we used $\beta^{1/5} < \delta^2/(10 r)$ and $2\gamma/(1+\gamma) \leq 1$ as $\gamma \leq 1$.
	\end{poc}
	
	With these sets $I_1,\dots, I_s, J_1,\dots, J_q$, we can distribute the vertices as desired. 
	Assume we already distributed vertices in $\bigcup_{i\in [kt]} Y_i$ to $X_1,\dots, X_{2s}$ and $X'_{p_{1}},\dots, X'_{p_q}$ in such a way that every vertex in $Y_{kt}^{B}$ is in $X_{2\ell'}$ for some $\ell'\in [s]$. If $k=0$, then we assume $\ell'=0$.
	If $k+1 \in I_{\ell}$, choose $\ell^*\in [s]$ such that $\ell'+\ell^* \in \{\ell,\ell+s\}$.
	We put vertices in $Y_{kt+1}^{A}, Y_{kt+1}^{B}, Y_{kt+2}^{A},\dots, Y_{kt +\ell^*}^{B}$ to $X_{2\ell'+1},X_{2\ell'+2},\dots, X_{2\ell-1}, X_{2\ell}$, respectively and we allocate the remaining vertices in $W^A_{k+1}\cup Z^A_{k+1}$ to $X_{2\ell-1}$ and the remaining vertices in $W^B_{k+1}\cup Z^B_{k+1}$ to $X_{2\ell}$.
	
	If $k+1 \in J_{\ell}$ for some $\ell\in [q]$, choose $\ell^*\in [s]$ such that $\ell'+\ell^* = p_{\ell}$. We allocate vertices in $Y_{kt+1}^{A}, Y_{kt+1}^{B}, Y_{kt+2}^{A},\dots, Y_{kt +\ell^*}^{B}$ to $X_{2\ell'+1},X_{2\ell'+2},\dots, X_{p_{\ell}-1}, X_{p_{\ell}}$ respectively and we allocate the remaining vertices in $W^A_{k+1}\cup Z^A_{k+1}$ to $X'_{p_{\ell}}$ and the remaining vertices in $W^B_{k+1}\cup Z^B_{k+1}$ to $X_{p_{\ell}}$.
	
	By repeating this for $k=0,\dots, t$, we distribute all vertices. For each $i\in [s]$ and $j\in [q]$, let $X_{u_i}=X_i$ and $X_{v_{p_j}}= X'_{p_j}$. 
	Then, by the bandwidth condition, all edges of $H$ are between $Y_i$ and $Y_{i+1}$, so we know that each edge of $H$ is between $X_u$ and $X_v$ for some $uv\in E(R_0)$. 
	As $\abs{W}<\delta^2 d/r$, we have
	for each $u=u_{2\ell-1},   v=v_{p_{\ell'}} \in V(R_0)$, the above distribution ensures that 
	\[\abs{X_{u}} \leq \Bigabs{\bigcup_{i\in I_\ell} Z_i^A} +\abs{W} \leq \frac{d}{r} \enspace \text{and} \enspace
	\abs{X_{v}} \leq \Bigabs{\bigcup_{i\in J_{\ell'}} Z_i^A} +\abs{W} \leq \frac{d}{r}.\]
	Again, as $\abs{W}<\delta^2 d/r$, for each $u=u_{2\ell}, u'=u_{p_{\ell'}} \in V(R_0)$ where $2\ell \notin \{p_1,\dots, p_{q}\}$, we have
	\[
	\abs{X_{u}} \leq \Bigabs{\bigcup_{i\in I_\ell} Z_i^B} +\abs{W} \leq  \frac{d}{r} , \enspace \text{and} \enspace
	\abs{X_{u'}} \leq \Bigabs{\bigcup_{i\in I_{p_{\ell'}}} Z_i^B \cup \bigcup_{i\in J_{\ell'} } Z_i^B } +\abs{W} \leq \frac{d}{r}.
	\]
	This proves the lemma.
\end{proof}

\subsection{Proof of Lemma~\ref{lem: dense}}
The last ingredient for the dense case is the following result of Voss and Zuluaga providing a long cycle in $2$-connected graphs.
\begin{lemma}\cite{VZ}\label{lem: long cycle}
	Suppose that a graph $R$ is a $2$-connected graph on at least $2d$ vertices with $\delta(G)\geq d$.
	If $R$ is a bipartite graph with vertex partition $A\cup B$, then $R$ has an even cycle of length at least $\min\{2\abs{A},2\abs{B},4d-4\}$.
	If $R$ is not a bipartite graph, then it has an odd cycle with length at least $\min\{\abs{R}-1, 2d-1\}$.
\end{lemma}

\begin{proof}[Proof of Lemma~\ref{lem: dense}]
	Let $d':=d/n>\eta$.
	We can choose numbers $r_1\in \mathbb{N}$ and $ \beta, \eps', \tau>0$ such that 
	\[0< 1/d \ll \alpha \ll \beta \ll 1/r_1 \ll \eps' \ll \tau \ll \eta, \delta, \eps < 1.\]
	
	We apply Lemma~\ref{lem: reg} to $G$ with $\eps',\tau,d'$ playing the roles of $\eps,\tau,d$ to obtain a partition $V_1\cup \dots \cup V_r$ of $V(G)$  and a corresponding reduced graph $R$ with $V(R)=[r]$, $1/\eps' \leq r \leq r_1$, such that $|V_i| = \frac{(1\pm\eps') n}{r}$ for each $i\in[r]$, $G[V_i,V_j]$ is $(\eps', \tau^+)$-regular for each $ij \in E(R)$,
	and $d(R) \geq (d' -2\tau)r$.

	As $d(R) \geq (d'-2\tau)r$, we can find a $2$-connected subgraph $R'$ of $R$ with	\[d(R')\geq  (d'-3\tau)r \enspace \text{ and } \enspace \delta(R') \geq \frac{1}{2}(d'-3\tau)r.\]
	One easy way to see such a graph $R'$ exists is to apply Lemma~\ref{lem-expander} to $R$ to obtain $R'$ with $\tau, 1/4$ playing the roles of $\eps_1, \eps_2$.
	
	Let $r_0:= (d'-4\tau)r$.	Now we will find a graph $R_0$ in $R'$ which is either an odd cycle of length at least $r_0$ or an $(2a,b)$-sun with $a+b\geq r_0$.
 This will provides a structure in $G$ suitable for us to use the blow-up lemma.
	
 If $R'$ is not a bipartite graph, we let $R_0$ be an odd cycle of length at least $\min\{|R'|-1,\allowbreak 2\delta(R')-1\}\ge r_0$ in $R'$ guaranteed by Lemma~\ref{lem: long cycle}.
	
	If $R'$ is bipartite with vertex bipartition $A\cup B$ and $\abs{A}\leq \abs{B}$, then we let $R_0$ be a $(2a,b)$-sun in $R'$ for some $a,b$ with $a+b\geq r_0$. We claim that such a sun exists. 	If $\abs{A}\geq r_0$, then Lemma~\ref{lem: long cycle} yields an even cycle $C$ of length at least $\min\{2\abs{A}, 4 \delta(R')-4\}\ge 2r_0$ in $R'$. This is a $(2r_0,0)$-sun as claimed.
	 
	If $\abs{A}< r_0$, then let  $p = r_0 -\abs{A}$. 
	As $\abs{A}\le \abs{B}$, we have
	\[(d'-3\tau)r\leq d(R') = \frac{2e(R')}{\abs{R'}}\le \frac{2\abs{A}\abs{B}}{\abs{A}+\abs{B}} \leq \min\{2\abs{A},\abs{B}\}.\]
	Thus, we have $p\le \frac{r_0}{2} =\frac{1}{2}(d'-4\tau)r$ and  $r_0\le (d'-3\tau)r\leq \abs{B}$.
	Now, as $\abs{A}=r_0-p$, we use Lemma~\ref{lem: long cycle} to find an even cycle $C$ of length at least $\min\{2\abs{A}, 4 \delta(R')-4\} \geq 2\abs{A} $ in $R'$, which contains all vertices in $A$. Then $\abs{B\setminus V(C)}=\abs{B}-\abs{A} > p$. As each vertex in $B\setminus V(C)$ has at least $\delta(R')\geq \frac{1}{2}(d'-3\tau)r > p$ neighbours in $A$, we can find a matching of size at least $\min\{\abs{B\setminus V(C)},p\} \geq p$ in $R'[A,B\setminus V(C)]$.
	This matching together with the cycle $C$ forms a $(2\abs{A},p)$-sun with $\abs{A}+p = r_0$ as claimed.

	Now, using Lemma~\ref{lem: H partition odd} or Lemma~\ref{lem: H partition even} with $\eps', R_0, (1-\eps/2)d$ playing the roles of $\delta, R_0, d$, respectively, we can partition $V(H)$ into $\{X_u: u\in R_0\}$ where for each $i\in [p]$ we have 
	\[\abs{X_i} \leq \frac{(1-\eps/2) d}{r_0} \leq  (1-\eps/3)\frac{n}{r} \leq (1-\eps')\abs{V_i},\]
	and every edge of $H$ is between $X_u$ and $X_{v}$ for some $uv\in E(R_{0})$. Hence, \ref{blow1}--\ref{blow3} in Theorem~\ref{thm: blow up} are all satisfied with $\eps',\tau, R_0$ playing the roles of $\eps,\tau, R_0$, respectively, to conclude that $G$ contains $H$ as a subgraph. This proves the lemma.
\end{proof}

\section{Subdivisions in robust expanders with medium density}\label{sec: medium}

In this section, our goal is to prove Lemma~\ref{lem: moderate}, which finds an $H$-subdivision in a robust expander with medium density. 

We first prove the following lemma, which bounds the number of high degree vertices in our expanders.

\begin{lemma}\label{lem: L small}
	Let $0<1/d\ll \eps_1,\eps_2\ll \eps, 1/\Delta<1$,  and $H$ be a graph with at most $(1-\eps)d$ vertices with $\Delta(H)\leq \Delta$ and $G$ is an $n$-vertex graph and 	\begin{equation*}\label{eq-setup}
		m:= \frac{2}{\eps_1}\log^3\bfrac{15n}{\eps_2d}, \quad \enspace 
		\text{ and } \quad \enspace L:= \{ v\in V(G) : d_G(v) \geq dm^{10} \}.
	\end{equation*} Suppose $G$ is an $(\eps_1,\eps_2d)$-robust-expander. If $\abs{L}\geq d$, then $G$ contains $H$ as a subdivision.
\end{lemma}
\begin{proof}
	Let $V(H) = \{x_1,\dots, x_{h}\}$ with $h\leq (1-\eps)d$ and $x_{a_1}x_{b_1},\dots, x_{a_{h'}}x_{b_{h'}}$ be an arbitrary enumeration of $E(H)$ with $h' = e(H) \leq \Delta (1-\eps)d$.
	
	Take  a set  $Z=\{v_1,\dots, v_{h}\}$ of $h$ distinct vertices in $L$. Then, for each $i\in [h]$, the set $X_i:= N(v_i)$ has size at least $d m^{10}$.
	Assume that we have pairwise internally disjoint paths $P_{1},\dots, P_{\ell}$ with $0\leq \ell< h'$ where $P_{j}$ is a path between $v_{a_j}$ and $v_{b_j}$ of length at most $2m$.
	
	Let $W_\ell = \bigcup_{j\in [\ell]} \Int(P_j)$ be the union of the interior vertices of the paths. As $\abs{W_{\ell}}+\abs{Z}\leq h' \cdot 2m+h \leq 4\Delta d m \le \rho(d m^{10})\cdot dm^{10}/4$,  we can apply Lemma~\ref{lem: path} to get an $X_{a_{\ell+1}},X_{b_{\ell+1}}$-path $P$ avoiding  $W_{\ell}\cup Z$ of length at most $m$. Extending $P$, we obtain a $v_{a_{\ell+1}},v_{b_{\ell+1}}$-path $P_{\ell+1}$ of length at most $m+2\le 2m$. By repeating this for $\ell=0,1,\dots, h'-1$ in order, we obtain $\bigcup_{j\in[h']}P_j$, which is an $H$-subdivision,  in $G$.	
\end{proof}

\begin{proof}[Proof of Lemma~\ref{lem: moderate}]
As $\frac{d}{n} \le \eta\ll \eps_1,\eps_2$, we have that $m^{100} = \bigl(\frac{2}{\eps_1} \log^{3}\bfrac{15n}{\eps_2d}\bigr)^{100} < \frac{n}{d}$. Hence
\begin{equation}\label{eq: dm100}
	d m^{100} < n.
\end{equation} 
To derive a contradiction, we assume that $G$ does not contain an $H$-subdivision.
Then by Lemma~\ref{lem: L small}, we have 
\begin{equation}\label{eq: L small}
	\abs{L}< d.
\end{equation}

For a contradiction, we will find an $H$-subdivision in $G-L$. For this purpose, we need not only that $L$ is small, but also that the graph $G-L$ is still relatively dense to contain an $H$-subdivision. We claim that
\begin{equation}\label{eq: G-L average degree}
	d(G-L) \geq \frac{d}{m^2}.
\end{equation}

This follows essentially from the robust expansion property of $G$. Indeed, otherwise a random vertex set $X$ of size $dm^2$ chosen uniformly at random from $G-L$ has expected degree sum $\mathbb{E}[\sum_{x\in X} d_{G-L}(x)] \leq \frac{d}{m^2} \cdot \abs{X}$. Hence, there exists a set $X\subseteq V(G- L)$ of size $dm^2$ with $\sum_{x\in X} d_{G-L}(x) \leq \frac{d}{m^2} \cdot \abs{X}$. Then, $F= \partial_{G-L}(X)$ has at most $ \frac{d}{m^2}\abs{X} \leq d(G) \rho(\abs{X})\abs{X}$ edges as $d(G)\geq \eps^2 d$, and $\rho(dm^2) \ge \rho(n)>\frac{1}{m}> \frac{1}{\eps^2 m^2}$ due to $\eps_1\ll\eps$. Note that, by definition of $F$, once we delete the edges of $F$ from $G$, the external neighbourhood of $X$ lies entirely in $L$, that is, $N_{G\setminus F}(X)\subseteq L$.
However, this implies 
\[\abs{N_{G\setminus F}(X)} \le \abs{L}\stackrel{\eqref{eq: L small}}{\leq} d \leq \rho(dm^2)\cdot dm^2 = \rho(\abs{X})\cdot \abs{X},\]
contradicting that $G$ is an $(\eps_1,\eps_2 d)$-robust-expander. Hence we have \eqref{eq: G-L average degree}.

Now that $G-L$ is still relatively dense and no vertex in $G-L$ has too large degree as $\Delta(G-L)\leq dm^{10}$, we can find in $G-L$ many webs with disjoint interiors.
\begin{claim}\label{cl: webs}
	The graph $G-L$ contains $(m^2, m^{10}, d/m^3, 4m)$-webs $W_1,\dots, W_{2d}$ where the interiors of the webs are pairwise disjoint.
\end{claim}
To find such webs, we follow the strategy as~\cite[Lemma 5.7]{KLShS17}. We include the proof in the online appendix.

Let $W_1,\dots, W_{2d}$ be the $(m^2, m^{10}, d/m^3, 4m)$-webs guaranteed by the claim. For each $i\in[2d]$, let $w_i$ be the core vertex of the web $W_i$ and let $C= \bigcup_{i\in [2d]} \Cen(W_i)\cup L$, and so
\[ \abs{\Ext(W_i)} =  d m^{9}\quad \enspace \text{and} \enspace \quad \abs{C} \leq  2d(m^2\cdot 4m+1)+d\le10dm^3.\]

Before moving on, we set up two pieces of notations. For a path $Q$ with endvertices $a\in \Ext(W_i)$ and $b\in \Ext(W_j)$, for some distinct $i, j\in [2d]$ such that $Q\cap (\Int(W_i)\cup \Int(W_j))=\varnothing$, we let $Q^* = Q\cup P(W_i,a)\cup P(W_j,b)$ be the $w_i,w_j$-path extending $Q$ in $W_i\cup W_j$. 

For a given set $Z$, we say a web $W$ is \emph{$Z$-good} if $\abs{\Int(W)\cap Z} \leq m^{12}/2$.
We define $(X,I,I',\mathcal{Q},f)$ be a \emph{good path system} if the followings hold.

\stepcounter{propcounter}
\begin{enumerate}[label = {\bfseries \Alph{propcounter}\arabic{enumi}}]
\item\label{Q1} $X\subseteq V(H)$ and $f:X\rightarrow [2d]$ is an injective map with $f(X)=I$.
\item\label{Q2} $\mathcal{Q}$ is a collection of paths $Q_{ij}$ for some $ij\in \binom{[2d]}{2}$ containing  $\{ Q_{f(e)}: e\in E(H[X])\}$ such that for each $ij\in \binom{[2d]}{2}$, $Q_{ij}$ is an $a_{i,j},a_{j,i}$-path of length at most $m$ where $a_{i,j}\in \Ext(W_{i})$ and $a_{j,i}\in \Ext(W_{j})$ with $Q_{ij}\cap \left(\Int(W_{i})\cup \Int(W_{j})\right) =\varnothing$.

\item\label{Q3} For each $x\in X$ and $N_H(x)\cap X =\{ y_1,\dots, y_s\}$, the only pairwise intersections between paths $P(W_{f(x)}, a_{f(x),f(y_1)}),\dots,\allowbreak P(W_{f(x)}, a_{f(x),f(y_s)})$ are at $w_{f(x)}$.
\item\label{Q4} $\{ Q_{ij}^*-C: Q_{ij}\in\cQ, ij\in \binom{[2d]}{2}\}$ is a collection of pairwise disjoint paths in $G-C$.
\item\label{Q5} $I'=\{i'\in [2d] : W_{i'} \text{ is not }V(\cQ)\text{-good}\}\subseteq [2d]\setminus I$.
\end{enumerate}

We shall show that there is a good path system with $X=V(H)$, which would finish the proof as by~\ref{Q1}--\ref{Q4}, $\bigcup_{e\in E(H)}Q_e^*$ is an $H$-subdivision.

We proceed as follows.

\renewcommand\thestep{$0$}\begin{step} Fix an arbitrary ordering $\sigma$ on $V(H)$, say the first vertex is $x_1$.  Let $X_1=\{x_1\}$, $f(x_1)=1$, $I_1:= \{1\}$, $I'_1:=\varnothing$ and $\mathcal{Q}_1=\varnothing$.
Then $(X_1,I_1,I'_1,\mathcal{Q}_1,f)$ is a good path system. Proceed to Step $1$.\end{step}

\renewcommand\thestep{$i$, $i\ge 1$}\begin{step} Stop if either $X_i= V(H)$, or $I_i\cup I'_{i}=[2d]$.
\begin{itemize}
	\item \emph{Add a new vertex}. 
	\begin{itemize}
		\item Let $x$ be the first vertex in $\sigma$ on $V(H)\setminus X_i$. Choose an unused $V(\cQ_i)$-good web $W_{i^*}$ with $i^*\in [2d]\setminus(I_i\cup I'_i)$ and let $f(x)=i^*$.
		
	   \item Find pairwise disjoint paths $Q_{f(x)f(y_1)},\dots, Q_{f(x)f(y_s)}$ satisfying \ref{Q2},\ref{Q3} and \ref{Q4} with respect to $X_{i}\cup\{x\}$, and add these paths to $\mathcal{Q}_{i}$ to obtain $\mathcal{Q}_{i+1}$.
	\end{itemize}
	
	\item \emph{Update bad webs}.
	\begin{itemize}
		\item Let $I'_{i+1}=\{i'\in [2d] : W_{i'} \text{ is not }V(\cQ_{i+1})\text{-good}\}$.
		
		\item Set $I_{i+1}= I_i\cup \{i^*\}\setminus I'_{i+1}$ and $X_{i+1} = f^{-1}(I_{i+1})$. 
		
		\item Replace $f$ with its restriction $f\mid_{X_{i+1}}$  on $X_{i+1}$.
	\end{itemize}
\end{itemize}
Proceed to Step $(i+1)$ with the new good path system $(X_{i+1},I_{i+1},I'_{i+1},\mathcal{Q}_{i+1},f)$.
\end{step}

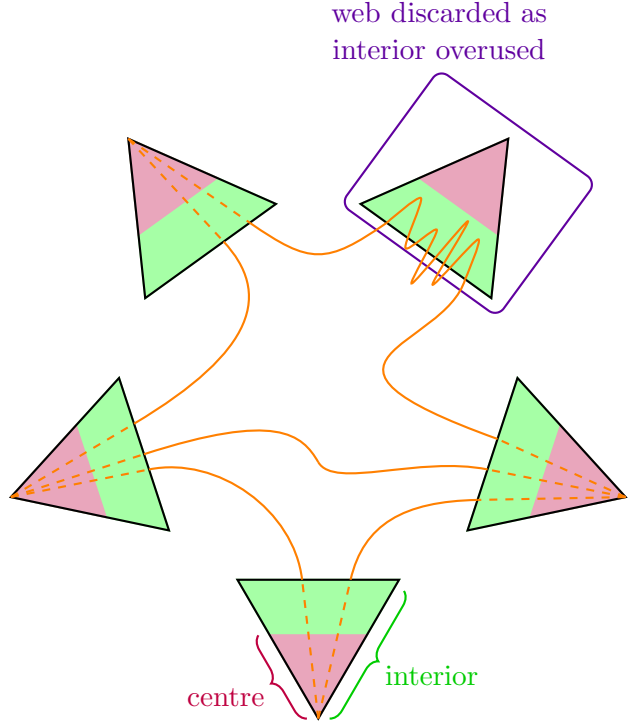
\begin{figure}[ht]
\begin{center}
\begin{tikzpicture}[thick, scale=.85]
\filldraw[green!35!white] (-90:5) -- +(60:2.5) -- +(120:2.5) -- cycle;
\filldraw[purple!35!white] (-90:5) -- +(60:1.5) -- +(120:1.5) -- cycle;
\draw (-90:5) -- +(60:2.5) -- +(120:2.5) -- cycle;
\path (-90:5) -- ++(120:2.5) -- ++(0:1) coordinate (a-to-e);
\path (-90:5) -- ++(120:3.75) -- ++(0:1.5) coordinate (a-to-e-control);
\draw[dashed, orange] (-90:5) -- (a-to-e);

\path (-90:5) -- ++(120:2.5) -- ++(0:1.75) coordinate (a-to-b);
\path (-90:5) -- ++(120:3.5) -- ++(0:2.45) coordinate (a-to-b-control);
\draw[dashed, orange] (-90:5) -- (a-to-b);

\draw [purple, decoration={brace,amplitude=5pt}, decorate] (-90:5) ++(180:5pt) -- ++(120:1.5) node [pos=0.65,anchor=north east,yshift=-5pt] {centre};
\draw [green!80!black, decoration={brace,mirror,amplitude=5pt}, decorate] (-90:5) ++(0:5pt) -- ++(60:2.3) node [pos=0.6,anchor=north west,yshift=-5pt] {interior};

\filldraw[green!35!white] (-18:5) -- +(-168:2.5) -- +(132:2.5) -- cycle;
\filldraw[purple!35!white] (-18:5) -- +(-168:1.5) -- +(132:1.5) -- cycle;
\draw (-18:5) -- +(-168:2.5) -- +(132:2.5) -- cycle;

\path (-18:5) -- ++(-168:2.5) -- ++(72:.5) coordinate (b-to-a);
\path (-18:5) -- ++(-168:3.75) -- ++(72:.75) coordinate (b-to-a-control);
\draw[dashed, orange] (-18:5) -- (b-to-a);

\path (-18:5) -- ++(-168:2.5) -- ++(72:1) coordinate (b-to-e);
\path (-18:5) -- ++(-168:3.75) -- ++(72:1.5) coordinate (b-to-e-control);
\draw[dashed, orange] (-18:5) -- (b-to-e);

\path (-18:5) -- ++(-168:2.5) -- ++(72:1.5) coordinate (b-to-c);
\path (-18:5) -- ++(-168:6.75) -- ++(72:3.75) coordinate (b-to-c-control);
\draw[dashed, orange] (-18:5) -- (b-to-c);

\filldraw[green!35!white] (54:5) -- +(-96:2.5) -- +(-156:2.5) -- cycle;
\filldraw[purple!35!white] (54:5) -- +(-96:1.5) -- +(-156:1.5) -- cycle;
\draw (54:5) -- +(-96:2.5) -- +(-156:2.5) -- cycle;
\draw[purple!50!blue, rounded corners] (54:5.25) +(144:1.5) node [anchor=south, align=center] {web discarded as\\interior overused} -- +(-156:3) -- +(-96:3) -- +(-36:1.5) -- cycle;

\path (54:5) -- ++(-96:2.5) -- ++(144:.5) coordinate (c-to-b);
\path (54:5) -- ++(-96:3.75) -- ++(144:.75) coordinate (c-to-b-control);
\path (54:5) -- ++(-96:2.5) -- ++(144:2) coordinate (c-to-d);
\path (54:5) -- ++(-96:3.75) -- ++(144:3) coordinate (c-to-d-control);

\filldraw[green!35!white] (126:5) -- +(-24:2.5) -- +(-84:2.5) -- cycle;
\filldraw[purple!35!white] (126:5) -- +(-24:1.5) -- +(-84:1.5) -- cycle;
\draw (126:5) -- +(-24:2.5) -- +(-84:2.5) -- cycle;

\path (126:5) -- ++(-24:2.5) -- ++(-144:.5) coordinate (d-to-c);
\path (126:5) -- ++(-24:3.75) -- ++(-144:.75) coordinate (d-to-c-control);
\draw[dashed, orange] (126:5) -- (d-to-c);

\path (126:5) -- ++(-24:2.5) -- ++(-144:1) coordinate (d-to-e);
\path (126:5) -- ++(-24:3.75) -- ++(-144:1.5) coordinate (d-to-e-control);
\draw[dashed, orange] (126:5) -- (d-to-e);

\filldraw[green!35!white] (-162:5) -- +(-12:2.5) -- +(48:2.5) -- cycle;
\filldraw[purple!35!white] (-162:5) -- +(-12:1.5) -- +(48:1.5) -- cycle;
\draw (-162:5) -- +(-12:2.5) -- +(48:2.5) -- cycle;

\path (-162:5) -- ++(-12:2.5) -- ++(108:1) coordinate (e-to-a);
\path (-162:5) -- ++(-12:3.75) -- ++(108:1.5) coordinate (e-to-a-control);
\draw[dashed, orange] (-162:5) -- (e-to-a);

\path (-162:5) -- ++(-12:2.5) -- ++(108:1.25) coordinate (e-to-b);
\path (-162:5) -- ++(-12:5) -- ++(108:2.5) coordinate (e-to-b-control);
\draw[dashed, orange] (-162:5) -- (e-to-b);

\path (-162:5) -- ++(-12:2.5) -- ++(108:1.75) coordinate (e-to-d);
\path (-162:5) -- ++(-12:5) -- ++(108:3.5) coordinate (e-to-d-control);
\draw[dashed, orange] (-162:5) -- (e-to-d);

\path (-90:1) coordinate (kink) -- ++(120:.5) coordinate (k1) -- ++(-60:1) coordinate (k2);
\draw[orange] (a-to-b) .. controls (a-to-b-control) and (b-to-a-control) .. (b-to-a);
\draw[orange] (a-to-e) .. controls (a-to-e-control) and (e-to-a-control) .. (e-to-a);
\draw[orange] (b-to-e) .. controls (b-to-e-control) and (k2) .. (kink);
\draw[orange] (e-to-b) .. controls (e-to-b-control) and (k1) .. (kink);
\draw[orange] (b-to-c) .. controls (b-to-c-control) and (c-to-b-control) .. (c-to-b);
\draw[orange] (d-to-c) .. controls (d-to-c-control) and (c-to-d-control) .. (c-to-d);
\draw[orange] plot [smooth] coordinates {(c-to-b) (45:3.5) (45:2.5) (50:3.6) (51:2.3) (55:3.2) (60:2.7) (63:3.5) (c-to-d)};
\draw[orange] (d-to-e) .. controls (d-to-e-control) and (e-to-d-control) .. (e-to-d);
\end{tikzpicture}
\end{center}
\caption{An example for $H=K_{1,1,2}$ (the diamond graph).}
\end{figure}

Let $(X,I,I',\mathcal{Q},f)$ be the good path system obtained from the  above  process. Note that the sequence $\abs{X_1},\abs{X_2},\dots$ might \emph{not} be an increasing sequence, as we may delete some elements when updating the list of bad webs in each step. However, we will show that this procedure will eventually end with $X=V(H)$ as desired.

First, we claim that $\abs{I'}\le d/m$. Note that $\cQ$ might contain some paths which connects $W_{i}$ with $i\in I'$. However, as at most $\Delta(H)\le \Delta$ paths are added at each step, 
\[\abs{\cQ}\le 2d\cdot \Delta(m+1)\le dm^2.\] 
Recall that $\{\Int(W_{j}): j\in [2d]\}$ are pairwise disjoint and so, by definition of $I'$, $\abs{I'}\cdot m^{12}/2\le \abs{\cQ}$. Thus,
$\abs{I'}\le \frac{\abs{\cQ} }{m^{12}/2} < d/m$ as claimed.

As in each step a new vertex is added before updating $I'_{i+1}$, the above claim guarantees that this process terminates when $\abs{I\cup I'}\le \abs{H}+\abs{I'}<2d$, implying that $X=V(H)$. To finish the proof, it is left to show that all connections in  each step, i.e.\ paths in $\cQ_{i+1}\setminus\cQ_i$, can indeed be constructed to keep the process running.

Let $x$, $i^*=f(x)$ and $\{y_1,\ldots, y_s\}=N_H(x)\cap (X_i\cup\{x\})$ be as in \textbf{Step} $i$, for some $i\ge 1$. Consider now  $j\in I_i\cup \{i^*\}=f(X_i\cup\{x\})$. Note that as $(X_{i},I_{i},I'_{i},\mathcal{Q}_{i},f)$ is a good path system at the beginning of this step, by~\ref{Q5},  $W_j$ is $V(\cQ_i)$-good. Also, as $V(\cQ_i)$ is disjoint from $C$ by~\ref{Q2} and \ref{Q4}, at most $\Delta(H)\le \Delta$ paths in $\Cen(W_j)$ are involved in previous connections $\cQ_i^*:=\{Q^*:  Q\in\cQ_i\}$. Thus, there  are  at  least  $(m^2-\Delta)m^{10}-m^{12}/2\ge m^{12}/4$ available paths in $\Int(W_j)\setminus\Cen(W_j)$ (and their corresponding paths in $\Cen(W_j)$) disjoint from $V(\cQ_i^*)$; let $A_j\subseteq \Ext(W_j)$ be the union of the leaves of the stars corresponding to these available paths. Then $\abs{A_j}\ge a:=dm^9/4$.

Now, for each $j\in[s]$, since 
\[\abs{C\cup \Int(W_{i^*})\cup \Int(W_{f(y_j)})\cup V(\cQ_i^*)}\le 10dm^3+20m\cdot m^{12}+30m\abs{\cQ}\le dm^4\le \rho(a)a/4,\]
using Lemma~\ref{lem: path}, we can find the desired path $Q_{xy_j}$ with length at most $m$ connecting $A_{i^*}$ and $A_{f(y_j)}$ avoiding $C\cup \Int(W_{i^*})\cup \Int(W_{f(y_j)})\cup V(\cQ_i^*)$.

This finishes the proof of Lemma~\ref{lem: moderate}.
\end{proof}

\section{Subdivisions in sparse robust expanders}\label{sec-sparse}
In this section, we prove Lemma~\ref{lem: very sparse}. First, we show  in Section~\ref{sec-still-dense} that we cannot have too many edges sticking out of a small set of vertices, for otherwise we can obtain $H$ even as a subgraph in $G$. In particular, $G-L$ is still dense (Claim~\ref{cl: ave deg after deleting L}), and consequently it must contain a subexpander $F$ with $d(F)=\Omega(d)$. 

Note that this subexpander could be a lot smaller and hence we have no control on its density. Suppose for a moment that $F$ is sparse, and additionally $F$ is large, then it inherits the `bounded' maximum degree from $G-L$. This  case is handled in Section~\ref{sec-sparse-regular}. For such an `almost regular' sparse expander, we can work entirely within it to find an $H$-subdivision (Lemma~\ref{lem: sparse regular bdd max deg}).

Now, if $F$ has medium edge density, then we can invoke Lemma~\ref{lem: moderate} on $F$ and we are done.  Therefore, $F$ must either be sparse and small, or very dense. In both cases, $F$ is small in order. To this end, we may assume that all subexpanders in $G-L$ are small. In Section~\ref{sec-many-small-subexpander}, we shall iteratively pull out many small expanders in $G-L$ that are pairwise far apart (\ref{F1}--\ref{F4}), most of which expand well inside $G-L$ (Claim~\ref{cl: one expand}). 

We then in Section~\ref{sec-making-nakji} link these nicely expanding small expanders to construct nakjis (Claim~\ref{cl-many-nakjis}). The strategy here is to extract nakjis iteratively. Every time, we step far away in $G-L$ from previously built nakjis and expand an appropriate collection of expanders to find the next one.

Finally, the finishing blow is delivered in Section~\ref{sec-nakji-subd}, in which we anchor on the nakjis and connect them to build an $H$-subdivision.

\subsection{\texorpdfstring{$G-L$}{G--L} still dense}\label{sec-still-dense}
As $d< m^{100} = \Bigl(\frac{2}{\eps_1} \log^3\bfrac{15n}{\eps_2d}\Bigr)^{100}$ and $1/d\ll 1/\Delta$, we have $n> 10 \Delta d^{\Delta}$. Suppose to the contrary that $G$ does not have any $H$-subdivision.
Then by Lemma~\ref{lem: L small}, $\abs{L}<d$.

We first establish the following claim, stating that the density does not drop much upon removing a small set of vertices. The idea is that if lots of edges are incident to a small set, then we will see a dense and skewed bipartite subgraph, which contains a copy of $H$.

\begin{claim}\label{cl: ave deg after deleting L}
	For any set $U\subseteq V(G-L)$ of size at most $n/m^{200}$, we have
	$d(G-L-U) \geq \eps d/6$.	
\end{claim}
\begin{poc}
	Let 
	\[Z:= \{ v\in V(G-L): \abs{N_{G}(v,L)}> (1-\eps)d/2\}.\]
	
	Suppose that $\abs{Z}> \Delta d^\Delta$.  For a given function $f:Z\rightarrow \binom{L}{\Delta}$, we let 
	\[t(f):= \sum_{Y\in \binom{L}{\Delta} } \min\{\abs{f^{-1}(Y)}, \Delta\}.\]
	Let $f: Z\rightarrow \binom{L}{\Delta}$ be a function with maximum $t(f)$ among all functions satisfying $f(v) \subseteq N_{G}(v,L)$ for every $v\in Z$. As $\abs{N_{G}(v,L)} \geq \Delta$, such a function $f$ must exist.
	However, as $\abs{Z}> \Delta d^{\Delta} >\Delta \binom{\abs{L}}{\Delta}$, there exist $\Delta+1$ vertices 
	$u_1,\dots, u_{\Delta+1} \in Z$ with $f(u_1)=\dots= f(u_{\Delta+1})$. Let $X= N_{G}(u_1,L)$. If any set $Y\in \binom{X}{\Delta}$ satisfies $\abs{f^{-1}(Y)}<\Delta$, then we can redefine $f(u_1)$ to be $Y$ to increase $t(f)$. Hence, by the maximality of $f$, for every set $Y\in \binom{X}{\Delta}$, there are  $\Delta$ distinct vertices $v_Y^{1},\dots, v_Y^{\Delta} \in Z$ such that $Y\subseteq N(v_Y^{i},L)$ for each $i\in [\Delta]$ and $v_Y^{i}\neq v_{Y'}^{j}$ for all $Y\neq Y' \in \binom{X}{\Delta}$ and $i,j\in [\Delta]$.
	It is easy to see that $X$ together with the vertices $\{v_Y^{i}: Y\in \binom{X}{\Delta},i\in [\Delta]\}$ induces a graph containing any $(1-\eps)d$-vertex bipartite graph with maximum degree at most $\Delta$. Hence $G$ contains $H$ as a subgraph, a contradiction. Thus $\abs{Z}\le  \Delta d^{\Delta}.$
	
	As $n\geq 10 \Delta d^{\Delta}$ and $\abs{Z\cup L}< 2\Delta d^{\Delta}$, at least $4n/5$ vertices in $V(G)-L-Z$ have $\delta(G) - \frac{1}{2}(1-\eps)d \geq \eps d/4$ neighbours in $G-L$.
	Hence, $e(G-L) \geq \frac{1}{2}( 4n/5) \cdot (\eps d/4 ) \geq \eps dn/10$.
	Then, $G-L-U$ has at least $e(G-L) - \Delta(G-L) \abs{U} \geq \eps dn/10 - dm^{10}\cdot (n/m^{200}) \geq \eps dn/12$ edges. 
\end{poc}

From this point on, we will work with $G':=G-L$. Recall that
\begin{equation}\label{eq: G' max deg}
\Delta(G')\leq d m^{10} \leq m^{110}.
\end{equation}

\subsection{`Bounded' degree sparse expander}\label{sec-sparse-regular}
As outlined at the start, since $G-L$ is still dense, it contains a subexpander $F$ at our disposal. We first take care of the case when this subexpander is sparse and `almost regular'.

We will use the following proposition to take many vertices that are pairwise far apart (\cite[Proposition 5.3]{LM17} taking $s=2000$).
\begin{prop}\label{prop: pairwise far}
	Suppose that $F$ is an $n$-vertex graph with $\Delta(F)\leq \log^{50000}{n}$, and $n$ sufficiently large.
	Then there is a set of at least $n^{1/5}$ vertices pairwise having distance at least $\frac{\log{n}}{100000\log\log{n}}$.
\end{prop}

\begin{lemma}\label{lem: sparse regular bdd max deg}
Let $0\ll 1/d\ll \eps_1,\eps_2\ll \eps,1/\Delta<1$ and let $H$ be a graph with at most $d$ vertices and $\Delta(H)\le \Delta$. Suppose $F$ is an $n$-vertex $(\eps_1,\eps_2d)$-robust-expander with $\delta(F)\ge \eps^2 d$. If $\Delta(F)\leq \log^{30000}{n}$, then $F$ contains an $H$-subdivision.
\end{lemma}
\begin{proof}
Let 
\[r:= (\log\log n)^5\quad \text{ and } \quad r' = \sqrt{\log n}.\]
As $d\le \delta(F)/\eps^2\le \Delta(F)/\eps^2\le \log^{30001}n$, Proposition~\ref{prop: pairwise far} implies that we can find vertices $v_1,\dots, v_h$, where $h=\abs{H}\leq d$, such that the distance between any two of them is at least $2r+2r'$.
Let $x_1,\dots, x_h$ be the vertices of $H$ and $e_1=x_{a_1}x_{b_1},\dots, e_{h'}=x_{a_{h'}}x_{b_{h'}}$ be the edges of $H$ where $h'=e(H)\le \Delta d/2$.

Suppose that we have $Q_{1},\dots, Q_{\ell}$ for some $0\leq \ell< h'$ such that:
\stepcounter{propcounter}
\begin{enumerate}[label = {\bfseries \Alph{propcounter}\arabic{enumi}}]
	\item\label{QQ1}  for each $i\in [\ell]$, $Q_{i}$ is a $v_{a_{i}},v_{b_{i}}$-path with length at most $2\log^4n$;
	\item\label{QQ2}  for distinct $i,j\in [\ell]$, the paths $Q_{i}$ and $Q_{j}$ are internally vertex disjoint;
	\item\label{QQ3}  for each $i\in [h]$, for the edges $\{e_{k_1}, \dots, e_{k_{s}}\} =\{ e_{k_j} : k_j\in [\ell], x_i\in e_{k_j}\}$ with $k_1 < \dots< k_s$, the paths $Q_{{k_1}},\dots, Q_{{k_s}}$  form consecutive shortest paths from $v_i$ in $B^r(v_i)$; and
	\item\label{QQ4} for any $i\in[h]$ and $j\in[\ell]$ with $x_i\notin e_j$, $B^r(v_i)$ and $V(Q_{j})$ are disjoint.
\end{enumerate}
Let 
\[W_1 =\bigcup_{i\in [\ell]} \Int(Q_{i}), \enspace\quad W_2:=\bigcup_{ \substack{i\in [h]:~ x_i\notin e_{\ell+1}}} B^r(v_i) \quad\text{ and } \quad W=W_1\cup W_2.\]
Note that $v_1,\dots, v_h$ are pairwise a distance at least $2r+2r'$ apart, so by \ref{QQ4} we have
\[\abs{B^r_{F-W}(v_{k_{\ell+1}})}=\abs{B^r_{F-W_1}(v_{k_{\ell+1}})}
= \abs{B^{r-1}_{F-W_1}( B^1_{F-W_1}(v_{k_{\ell+1}}))}\] for each $k\in \{a,b\}$. Note that $\abs{B^1_{F-W_1}(v_{k_{\ell+1}})} \geq \eps^2 d - \Delta \geq \eps^2 d/2$ by \ref{QQ3} and \ref{QQ4}. By \ref{QQ3}, we can apply
Proposition~\ref{prop: expansion after deleting shortest path} with $B^1_{F-W_1}(v_{k_{\ell+1}}), W_1, \varnothing,\Delta$ playing the roles of $X, P, Y,q$, and then for each $k\in \{a,b\}$ we have 
\[\abs{B^r_{F-W}(v_{k_{\ell+1}})}=\abs{B^r_{F-W_1}(v_{k_{\ell+1}})} \geq \exp( (r-1)^{1/4}) \geq d \log^{8}{n},\]
where the last inequality follows from $d\le \log^{30001}n$.
This implies that \[\abs{W_1} \leq h'\cdot 2\log^4n \leq \Delta d\log^4n < \frac{1}{4} \rho(\abs{B^r_{F-W}(v_{k_{\ell+1}})}) \cdot \abs{B^r_{F-W}(v_{k_{\ell+1}})}.\]
Hence, by applying Proposition~\ref{prop: expansion after deleting shortest path}, now with $B^r_{F-W}(v_{k_{\ell+1}}), \varnothing, W_1$ playing the roles of $X,P,Y$ for each $k\in \{a,b\}$, we similarly have 
\[\abs{B^{r+r'}_{F-W}(v_{k_{\ell+1}})} =\abs{B^{r+r'}_{F-W_1}(v_{k_{\ell+1}})} \geq \exp((r')^{1/4}) \geq 
\exp(\sqrt[9]{\log n}).\]
 As $\Delta(F)\leq  \log^{30000}{n}$ and $d\le \log^{30001}n$, we then have
\[
\abs{W}\leq \abs{W_1} + \abs{W_2} \leq \Delta d\log^4{n} + d \cdot 2 (\log^{30000}n)^r < \frac{1}{4}\rho\left(\exp(\sqrt[9]{\log n})\right) \exp(\sqrt[9]{\log n}).
\]
Therefore, by Lemma~\ref{lem: path}, there is a path in $F-W$ of length at most $\log^4n$ between $B^{r+r'}_{F-W}(v_{a_{\ell+1}})$ and $ B^{r+r'}_{F-W}(v_{b_{\ell+1}})$. So we can let $Q_{{\ell+1}}$ be a shortest path between $v_{a_{\ell+1}}$ and $v_{b_{\ell+1}}$ in $F-W$, which has length at most $\log^4n+2r+2r'\le 2\log^4n$. Hence, $Q_1,\ldots, Q_{\ell+1}$ satisfy~\ref{QQ1}--\ref{QQ4}. Repeating this for $\ell=0,1,\dots, h'-1$, then the union of all paths $\bigcup_{i\in[h']}Q_i$ is an $H$-subdivision.
\end{proof}

\subsection{Many small subexpanders}\label{sec-many-small-subexpander}
With Lemma~\ref{lem: sparse regular bdd max deg} at hand, we can now proceed to show that all subexpanders in $G'=G-L$ must be small and that we can find many of them pairwise far apart.

Let $\mathcal{F}$ be a maximal collection of subgraphs of $G'$ satisfying the following.
\stepcounter{propcounter}
\begin{enumerate}[label = {\bfseries \Alph{propcounter}\arabic{enumi}}]
\item\label{F1}  For each $F\in \mathcal{F}$, $F$ is an $(\eps_1,\eps_2 d)$-expander with $d(F)\geq \eps d/10$ and $\delta(F)\geq \eps d/20$ and $F$ is $\eps^2 \nu d-$connected.
\item\label{F2}  For distinct $F,F'\in \mathcal{F}$, we have $B^{\sqrt{\log{n}}}_{G'}(V(F)) \cap B^{\sqrt{\log{n}}}_{G'}(V(F')) = \varnothing$.
\end{enumerate}
For each $F\in \mathcal{F}$, let 
\[n_F = \abs{F}, \quad\enspace  m_F = \frac{2}{\eps_1} \log^{3}\left(\frac{15n_F}{\eps_2d}\right) \enspace \quad\text{and} \quad\enspace U =\bigcup_{F\in\cF} B_{G'}^{2\sqrt{\log{n}}}(V(F)).\]
If some $F\in \mathcal{F}$ satisfies $ m_F^{100} < d(F) < \sqrt{n_F}$, then $F$ satisfies the conditions on Lemma~\ref{lem: moderate}. Hence Lemma~\ref{lem: moderate} yields an $H$-subdivision in $F$, a contradiction. Thus either
\[ n_F \le d(F)^2 \leq  \Delta(G')^2 \stackrel{\eqref{eq: G' max deg}}{\leq} m^{220}  \enspace \quad \text{or} \quad \enspace d(F) \leq m_F^{100}.\]
If the latter case holds, we claim that $n_F\leq \exp(\sqrt[50]{\log{n}})$. Indeed, if 
$n_F> \exp(\sqrt[50]{\log{n}})$ holds, then we have 
\[\Delta(F) \leq \Delta(G') \stackrel{\eqref{eq: G' max deg} }{\leq} m^{110} \leq \log^{30000}n_F.\] 
So we can apply Lemma~\ref{lem: sparse regular bdd max deg} on  $F$ to get an $H$-subdivision, a contradiction.
Thus we have
\begin{enumerate}[resume*]
	\item\label{F3}  for each $F\in \mathcal{F}$, $n_F \leq \max\{m^{220}, \exp(\sqrt[50]{\log{n}}) \} = \exp(\sqrt[50]{\log{n}}).$
\end{enumerate}
Moreover, 
if $\abs{\mathcal{F}} < n^{0.99}$, then by \eqref{eq: G' max deg} this implies that 
\[\abs{U}\leq n^{0.99}\cdot \exp(\sqrt[50]{\log{n}})\cdot 2\Delta(G')^{2\sqrt{\log{n}}} \leq  \frac{n}{m^{200}}.\] Hence, Claim~\ref{cl: ave deg after deleting L} implies that 
$G'-U$ has average degree at least $\eps d/6$.
Thus Lemma~\ref{lem-expander} finds another expander $F$ satisfying \ref{F1} in $G'-U$. Then \ref{F2} also holds, contradicting the maximality of $\cF$. Hence, we have
\begin{enumerate}[resume*]
	\item\label{F4} $\abs{\mathcal{F}}\geq n^{0.99}$.
\end{enumerate}

Futhermore, the following claim states that most of the expanders in $\cF$ expand well in $G'$.

\begin{claim}\label{cl: one expand}
There exist at least $\abs{\mathcal{F}}-d\log n$ graphs $F$ in $\mathcal{F}$ such that for each $1\leq r \leq \log{n}$, 
$$\abs{B^{r}_{G'}(V(F))} \geq  \frac{1}{d}\exp(r^{1/4}).$$
\end{claim}
\begin{poc}
Fix a choice of $1\le r\le \log n$ and consider a set $\mathcal{I}\subseteq \mathcal{F}$ with $\abs{\mathcal{I}}= d$ and let $X=\bigcup_{F\in \mathcal{I}} V(F)$.
Note that by~\ref{F1}, $\abs{X} \geq \eps d^2/10$, and so $\abs{L}<d\le  \frac{1}{4}\rho(\abs{X})\abs{X}$.
By Proposition~\ref{prop: expansion after deleting shortest path} with $L$ playing the role of $Y$, we have 
\[\abs{B^{r}_{G'}(X)} \geq \exp( r^{1/4}).\]
Then by the pigeonhole principal, there exists $F\in \mathcal{I}$ such that 
\[\abs{B^{r}_{G'}(V(F))} \geq \frac{1}{d}\exp(r^{1/4}).\]
Therefore, whenever there are $d$ members left in $\cF$, we can keep picking out one, the $r$-ball around which expands nicely in $G'$. Now varying $r$, we see that there are at least $\abs{\cF}-d\log n$ graphs $F$ in $\cF$ as claimed.
\end{poc}

By losing a factor of $2$ in size, i.e.\ instead of~\ref{F4}, $\abs{\cF}\ge n^{0.99}/2$, we may assume now every member in $\cF$ expands well in $G'$ as in Claim~\ref{cl: one expand}

\subsection{Linking subexpanders to nakjis}\label{sec-making-nakji}
We will now connect these far apart nicely expanding subexpanders in $\cF$ to obtain $d$ separate $(\Delta,\exp(\sqrt[50]{\log n}),m,2\sqrt{\log n})$-nakjis to anchor.

More precisely, take the following collections with maximum possible $p\in\mathbb{N}$.
\stepcounter{propcounter}
\begin{enumerate}[label = {\bfseries \Alph{propcounter}\arabic{enumi}}]
	\item \label{C1} $\mathcal{C}:=\{ C_{i,j} : i\in [p], 0\leq j\leq \Delta \}$ is a collection of vertex sets of distinct graphs in $\mathcal{F}$.
	\item \label{C2} For each $(i,j)\in [p]\times [\Delta]$, we have a vertex set $S_{i,j}$ with $C_{i,j}\subseteq S_{i,j} \subseteq B^{\sqrt[10]{\log{n}}}_{G'}(C_{i,j})$ with $\abs{S_{i,j}} = \exp(\sqrt[50]{\log{n}})$ and $G[S_{i,j}]$ is connected.
	\item \label{C3} $\{P_{i,j} : (i,j)\in [p]\times[\Delta]\}$ is a collection of pairwise vertex disjoint paths, each of length at most $10m$, such that $P_{i,j}$ is a $u_{i,j},v_{i,j}$-path where $u_{i,j}\in C_{i,0}$ and $v_{i,j}\in S_{i,j}$ and the internal vertices of $P_{i,j}$ are not in $C_{i,0}\cup S_{i,j}$.
	\item \label{C4} For distinct $(i,j), (i',j')\in [p]\times [\Delta]$, $P_{i,j}$ is disjoint from $S_{i',j'}$ and $\bigcup_{i''\in [p]\setminus \{i\}}C_{i'',0}$.
\end{enumerate}

So, by~\ref{F1} and~\ref{C1}--\ref{C4}, we see that for each $i\in[p]$, $\bigcup_{0\le j\le \Delta}C_{i,j}$ and $\bigcup_{j\in[\Delta]}P_{i,j}$ form a $(\Delta,\exp(\sqrt[50]{\log n}),m,2\sqrt{\log n})$-nakji, in which $C_{i,0}$ is the head and each $C_{i,j}$, $j\in[\Delta]$, is a leg.

\begin{claim}\label{cl-many-nakjis}
We have $p\geq d$.	
\end{claim}
\begin{poc}
Suppose to the contrary that $p<d$. Let $W$ be the set of vertices involved in $\mathcal{C}$ and all the paths $P_{i,j}$, and $W'$ be a $\sqrt{\log n}$-ball in $G'$ around $W$. That is,
\[W=\bigcup_{(i,j)\in [p]\times [\Delta]} V(P_{i,j})  \cup \bigcup_{C\in\mathcal{C}}V(C), \quad \text{ and } \quad W'=B_{G'}^{\sqrt{\log n}}(W).\]
Then, by~\ref{C1}--\ref{C3} and \eqref{eq: G' max deg}, we see that
\[\abs{W}\le (\Delta+1)d \cdot \bigl(\exp\bigl(\sqrt[50]{\log n}\bigr)+20m\bigr)\le \exp\bigl(\sqrt[49]{\log n}\bigr),\]
and
\[\abs{W'}\le \exp\bigl(\sqrt[49]{\log n}\bigr)\cdot 2\Delta(G')^{\sqrt{\log n}}\le n^{0.1}.\]
We shall find a nakji in $G'-W'$, which will lead to a contradiction to the maximality of $p$.

As $\abs{\mathcal{F}} -\abs{W'}> n^{0.98}$, we can choose two disjoint collections $\cF_0,\cF'\subseteq \cF$, containing subexpanders disjoint from $W'$, with $\abs{\cF_0}=d$ and $\abs{\cF'}=n^{0.97}$.  Let $X= \bigcup_{F\in \mathcal{F}_0} V(F)$, then by \ref{F1}, we have $\abs{X}\geq \frac{\eps d^2}{10}$.

Now assume that, for some $0\leq \ell\leq \Delta d$, we have pairwise vertex disjoint paths $Q_{1},\dots, Q_{\ell}$ in $G'$ satisfying the following for each $i\in [\ell]$.
\stepcounter{propcounter}
\begin{enumerate}[label = {\bfseries \Alph{propcounter}\arabic{enumi}}]
\item\label{Q'1} $Q_i$ is a path of length at most $10m$ from $X$ to $F'_i\in \mathcal{F}'$.
\item\label{Q'2} $Q_1,\dots, Q_\ell$  are consecutive shortest paths from $X$ in $B_{G'}^{\sqrt{\log{n}}}(X)$.
\item\label{Q'3} $Q_i$ is disjoint from $W \cup \bigcup_{j \in [\ell]\setminus\{i\}} B^{\sqrt[10]{\log{n}}}_{G'}(V(F'_j))$ and it is a shortest path from $V(F_i')$ in $B^{\sqrt[10]{\log{n}}}_{G'}(V(F'_i))$.
\end{enumerate} 
Let 
\[Q = \bigcup_{i\in [\ell]} V(Q_i), \quad Q'=B^{\sqrt[10]{\log{n}}}_{G'}(Q) \quad\text{and}\quad W^*=\bigcup_{i\in [\ell]} B^{\sqrt[10]{\log{n}}}_{G'}(V(F'_i)),\]
then we have by \ref{F3} and \eqref{eq: G' max deg} that
\begin{gather*}\abs{Q} \leq d^2m,\quad
\abs{Q'}\le d^2m\cdot 2\Delta(G')^{\sqrt[10]{\log{n}}}\le \exp\bigl(\sqrt[9]{\log{n}}\bigr)\quad\text{and}\\
\abs{W^*} \leq \ell \exp\bigl(\sqrt[50]{\log{n}}\bigr)\cdot 2\Delta(G')^{\sqrt[10]{\log{n}}} \leq \exp\bigl(\sqrt[9]{\log{n}}\bigr).\end{gather*}
Let $\mathcal{F}''$ be the graphs in $\mathcal{F}'\setminus\{F'_1,\dots, F'_{\ell}\}$ which do not intersect with $Q'$. Then, letting
$U = \bigcup_{F\in \mathcal{F}''} V(F)$,
we have
\[\abs{U} \geq \abs{\mathcal{F}'}-\ell -\abs{Q'} \geq n^{0.97} - d^2 -  \exp(\sqrt[6]{\log{n}})\geq n^{0.9}.\]

We shall connect $X$ and $U$. First we expand $X$ as follows. As $\abs{X}\geq \eps d^2/10$, $\abs{L}\leq d$ and $\ell \leq \Delta d\le \frac{\abs{X}}{\log^8\abs{X}}$, \ref{Q'2} implies that we can apply Proposition~\ref{prop: expansion after deleting shortest path} with $X, Q, L$ playing the roles of $X, P, Y$ to obtain that 
\[\bigabs{B^{\sqrt{\log{n}}}_{G'-Q-W-W^*}(X)} = \bigabs{B^{\sqrt{\log{n}}}_{G'-Q}(X)}\geq \exp(\sqrt[8]{\log{n}}),\]
where the first equality follows from $X$ being far from $W\cup W^*$ owing to $X\cap W'=\varnothing$ and~\ref{F2}.

As 
\[\abs{L\cup Q\cup W\cup W^*} \leq 2\exp(\sqrt[9]{\log{n}})  < \frac{1}{4}\rho(\exp(\sqrt[8]{\log{n}}))\cdot \exp(\sqrt[8]{\log{n}})\]
and $\abs{U}\geq n^{0.9}$, Lemma~\ref{lem: path} implies that there exists a path of length at most $m$ between the sets $U$ and $B^{\sqrt{\log{n}}}_{G'-Q-W-W^*}(X)$ avoiding $L\cup Q\cup W\cup W^*$.
Let $Q'_{\ell+1}$ be a shortest such path with endvertices say $v\in B^{\sqrt{\log{n}}}_{G'-Q-W-W^*}(X)$ and $u\in U$. Let $F'_{\ell+1}\in \mathcal{F}''$ be the graph containing $u$. Appending a shortest path in $G'-Q$ from a vertex of $X$ to $v$ in $B^{\sqrt{\log{n}}}_{G'-Q}(X)$ to the path $Q'_{\ell+1}$, let the resulting path be $Q_{\ell+1}$.  The choices of $W^*$, $F_{\ell+1}'$ and the path $Q_{\ell+1}$ ensure that \ref{Q'1}--\ref{Q'3} hold for $Q_1,\dots, Q_{\ell+1}$.

\begin{figure}[ht]
\begin{center}
\begin{tikzpicture}[thick, scale=0.85]
\draw (2.5,4) ellipse [x radius=3, y radius=.75];
\node at (5.5,4) [anchor=south west] {$W$};
\draw[purple] (-1.5,2.5) .. controls (0,1.5) and (5.5,1.5) .. (6.5,3.5) coordinate[pos=.5] (w1);
\node[purple] at (6.5,3.5) [anchor=west] {$W'$};
\draw[dotted] (2.5,3.25) -- (w1) node [pos=.6, anchor=south west] {$\scriptstyle\sqrt{\log n}$};

\filldraw[cyan, fill=cyan!25!white] (0,0) circle (0.5);
\filldraw[cyan, fill=cyan!25!white] (2,0) circle (0.5);
\filldraw[cyan, fill=cyan!25!white] (4,0) circle (0.5);
\filldraw[cyan, fill=cyan!25!white] (6,0) circle (0.5);

\draw[ultra thick, blue] (0,-1) -- (6,-1);
\draw[ultra thick, blue] (0,1) -- (6,1);
\draw[ultra thick, blue] (6,1) arc [start angle=90, end angle=-90, radius=1];
\draw[ultra thick, blue] (0,1) arc [start angle=90, end angle=270, radius=1];
\node[blue] at (7,1) [anchor=south west] {$B_{G'-Q-W-W^*}^{\sqrt{\log n}}(X)$};
\draw[blue,->] (7.2,1.2) -- (6.8,.8);
\node[cyan] at (7.5,0) [anchor=west] {$X$};
\draw[cyan,->] (7.5,0) -- (6.6,0);

\filldraw[orange, fill=orange!25!white] (-1,-3) circle (0.5);
\node[orange] at (-1,-3.5) [anchor=north] {$F'_1$};
\filldraw[orange, fill=orange!25!white] (.5,-3) circle (0.5);
\node[orange] at (0.5,-3.5) [anchor=north] {$F'_2$};
\filldraw[orange, fill=orange!25!white] (2,-3) circle (0.5);
\node[orange] at (2,-3.5) [anchor=north] {$F'_3$};
\filldraw[orange, fill=orange!25!white] (3.5,-3) circle (0.5);
\node[orange] at (3.5,-3.5) [anchor=north] {$F'_4$};
\filldraw[orange, fill=orange!25!white] (5,-3) circle (0.5);
\node[orange] at (5,-3.5) [anchor=north east] {$F'_\ell$};
\draw[dotted, green!80!black] (5,-3.5) -- (5,-4.5) node [pos=.5, anchor=west] {$\!\!\scriptstyle\sqrt[10]{\log n}$};
\draw[ultra thick, dashed, green!80!black] (-1,-2) -- (5.5,-2);
\draw[ultra thick, dashed, green!80!black] (-1,-4.5) -- (5.5,-4.5);
\draw[ultra thick, dashed, green!80!black] (5.5,-2) arc [start angle=90, end angle=-90, radius=1.25];
\draw[ultra thick, dashed, green!80!black] (-1,-2) arc [start angle=90, end angle=270, radius=1.25];

\filldraw[orange, fill=orange!25!white] (8.5,-3) circle (0.5);
\filldraw[orange, fill=orange!25!white] (10,-3) circle (0.5);
\draw[ultra thick, purple, rounded corners] (10.5,-3.75) -- (7.5,-3.75) -- (7.5,-2.25) -- (10.5,-2.25);
\node[orange] at (9.25,-4) [anchor=north] {$\mathcal F''$};
\draw[orange,->] (9.25,-4) -- (9.25,-3);
\draw[purple,<-] (10,-2) -- (10,-1.5) node [anchor=south] {$U$};

\path (6,0) ++(-45:1) coordinate (q1);
\draw[red] plot [smooth] coordinates {(q1) (7,-1.25) (8,-1.5) (8.5,-2.25)};
\node[red, anchor=south] at (7.5,-1.4) {$Q'_{\ell+1}$};

\draw[purple!50!blue] (6,-.5) .. controls (6,-1.5) and (5,-1.5) .. (5,-2.5) node[pos=.4, anchor=west] {$Q_\ell$};
\draw[purple!50!blue] (2,-.5) .. controls (2.5,-1.5) and (2,-1) .. (2,-2.5) node[pos=.08, anchor=west] {$Q_3$};
\filldraw[white] (2.1,-1.3) rectangle (2.2,-1.5);
\draw[purple!50!blue] (-120:.5) .. controls (-120:1.5) and (0,-1.5) .. (-1,-2.5) node[pos=.6, anchor=east] {$Q_1$};
\draw[purple!50!blue] (0,-.5) .. controls (-.5,-1.5) and (.75,-1.5) .. (.5,-2.5) node[pos=.6, anchor=west] {$Q_2$};
\draw[purple!50!blue] (-50:.5) .. controls (-50:1) and (3,-1.5) .. (3.5,-2.5) node[pos=.7, anchor=west] {$Q_4$};
\end{tikzpicture}
\end{center}
\caption{The proof of Claim \ref{cl-many-nakjis}.}
\end{figure}
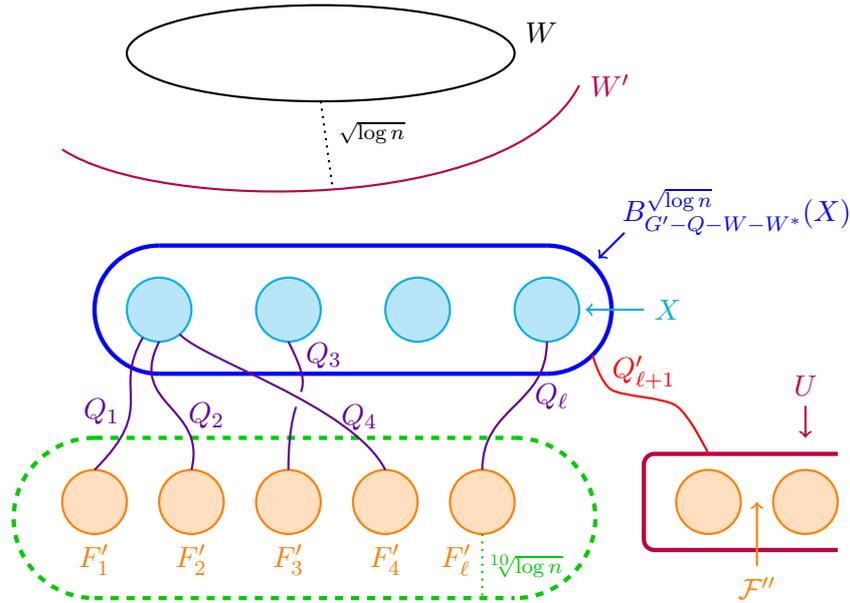

Repeating this for $\ell=0,\dots, d\Delta$, we obtain paths $Q_1,\dots, Q_{d\Delta}$. Recall that $\abs{\cF_0}=d$ and so by the pigeonhole principle, there exists a graph $F_0\in \mathcal{F}_0$ such that at least $\Delta$ paths among $\{Q_1,\dots, Q_{\Delta d}\}$ are incident to $F_0$. Relabelling and keeping the relative ordering, let those paths be $Q_1,\dots, Q_{\Delta}$ connecting $F_0$ and $F'_{1},\dots, F'_{\Delta} \in \mathcal{F}''$. Let $C_{p+1,0} = F_0$ and $C_{p+1,i} = F'_{i}$ for each $i\in [\Delta]$. 

Now fix $i\in [\Delta]$. By \ref{Q'3} and the definition of $\cF''\ni F'_i$,  $B^{\sqrt[10]{\log{n}}}_{G'}(V(F'_i))$ is disjoint from $Q\setminus V(Q_i)$ and $W$, and furthermore, by  Claim~\ref{cl: one expand}, we can find a connected subgraph $S_{p+1,i} \subseteq B^{\sqrt[10]{\log{n}}}_{G'}(V(F'_i))$ containing $V(F'_i)$ which satisfies~\ref{C2}. Let $v_{p+1,i}\in S_{p+1,i}$ be its first contact point with $Q_i$, and let $P_{p+1,i}=Q_i-S_{p+1,i}\setminus\{v_{p+1,i}\}$ be the truncated path. It is routine to check that \ref{C1}--\ref{C4} still hold with the additions of $C_{p+1,0}$, and $C_{p+1,i}\subseteq S_{p+1,i}$ and $P_{p+1,i}$, $i\le [\Delta]$. This contradicts the maximality of $p$. Hence, we have $p\geq d$, proving the claim.
\end{poc}

\subsection{The finishing blow}\label{sec-nakji-subd}
It is time now to complete the game: we will wire nakjis together to build an $H$-subdivision.

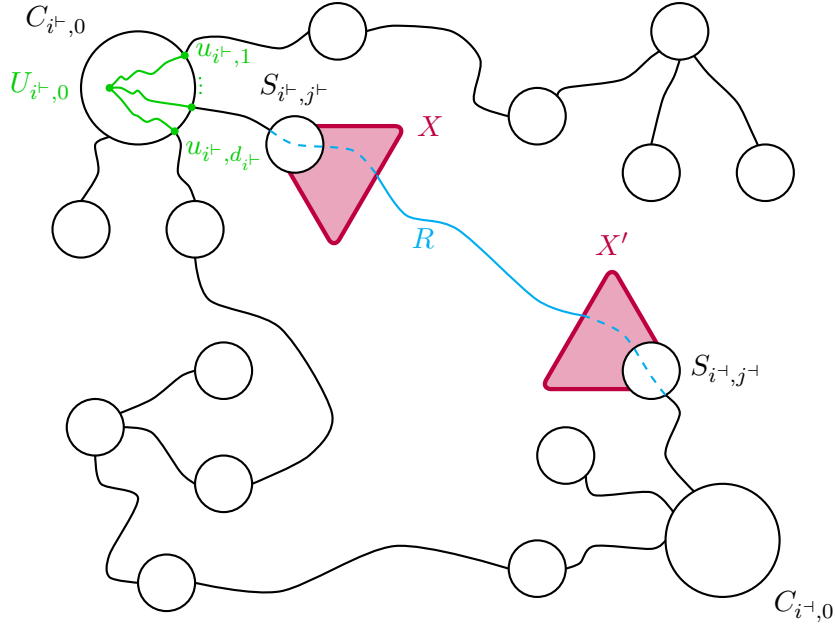
\begin{figure}[ht]
\begin{center}
\begin{tikzpicture}[thick, scale=0.75]
\filldraw[purple, ultra thick, fill=purple!35!white, rounded corners] (2.75,-1) ++(150:.65) -- ++(2.5,0) node [anchor=west] {$X$} -- ++(-120:2.5) -- cycle; 
\filldraw[white] (2,-.5) circle (0.5);
\filldraw[purple, ultra thick, fill=purple!35!white, rounded corners] (9,-5) ++(-30:.65) -- ++(-2.5,0) -- ++(60:2.5) node [anchor=south] {$X'$} -- cycle;
\filldraw[white] (9.5,-5.5) circle (0.5);

\draw (0,0) circle (1);
\draw (3.5,1) circle (0.5);
\filldraw[fill=white] (2.75,-1) circle (0.5);
\path (2.75,-1) ++(150:.5) coordinate (u2);
\draw (1,-2.5) circle (0.5);
\draw (-1,-2.5) circle (0.5);
\node at (135:1) [anchor=south east] {$C_{i^{\vdash},0}$};
\node at (2.75,-.5) [anchor=south] {$S_{i^\vdash,j^\vdash}$};
\draw plot [smooth] coordinates {(35:1) (1,.85) (1.5,.9) (2.5,.7) (3,1)};
\draw plot [smooth] coordinates {(-20:1) (1.85,-.5) (u2)};
\draw plot [smooth] coordinates {(-50:1) (.75,-1) (.7,-1.5) (.9,-1.75) (1,-2)};
\draw plot [smooth] coordinates {(-120:1) (-.75,-1) (-.7,-1.5) (-1,-2)};
\draw [dotted, green!80!black] (10:1.1) arc [start angle=10, end angle=-10, radius=1.1];

\node[green!80!black] at (-1,0) [anchor=east] {$U_{i^\vdash,0}$};
\filldraw [green!80!black] (-.5,0) circle (0.05);
\filldraw [green!80!black] (35:1) circle (0.05) node [anchor=west] {$u_{i^\vdash,1}$};
\filldraw [green!80!black] (-20:1) circle (0.05);
\filldraw [green!80!black] (-50:1) circle (0.05) node [anchor=north west] {$u_{i^\vdash,d_{i^\vdash}}$};
\draw[green!80!black] plot [smooth] coordinates {(-.5,0) (-.4,.1) (-.3,.2) (-.2,.15) (0,.3) (0.25,.25) (0.5,0.45) (35:1)};
\draw[green!80!black] plot [smooth] coordinates {(-.5,0) (-.3,0) (-.1,-.1) (0,-.05) (0.15,-.2) (0.35,-.25) (0.7,-.3) (-20:1)};
\draw[green!80!black] plot [smooth] coordinates {(-.5,0) (-.3,-.2) (-.1,-.45) (0,-.5) (0.15,-.6) (0.35,-.55) (-50:1)};

\draw (10.25,-8) circle (1);
\draw (7,-8.5) circle (.5);
\draw (7.5,-6.5) circle (.5);
\filldraw [fill=white] (9,-5) circle (.5);
\path (9,-5) ++(-60:.5) coordinate (xx1);
\path (10.25,-8) ++(120:1) coordinate (uu1);
\path (10.25,-8) ++(150:1) coordinate (uu2);
\path (7.5,-6.5) ++(-45:.5) coordinate (uu3);
\draw plot [smooth] coordinates {(uu1) (9.25,-6.5) (9.5,-5.75) (xx1)};
\draw plot [smooth] coordinates {(9.25,-8) (9,-8.15) (8,-8.1) (7.75,-8.4) (7.5,-8.5)};
\draw plot [smooth] coordinates {(uu2) (9,-7.15) (8,-7.2) (uu3)};
\node at (9.5,-5) [anchor=west] {$S_{i^\dashv,j^\dashv}$};
\path (10.25,-8) ++(-45:1) node [anchor=north west] {$C_{i^{\dashv},0}$};

\path (2.75,-1) ++(150:.65) -- ++(2.5,0) -- ++(-120:1) coordinate (x1);
\draw[cyan,dashed] plot [smooth] coordinates {(u2) (2.75,-1) (3.75,-1.1) (x1)};
\path (9,-5) ++(-30:.65) -- ++(-2.5,0) -- ++(60:1.5) coordinate (xx2); 
\draw[cyan,dashed] plot [smooth] coordinates {(xx1) (9,-5.1) (8.5,-4.35) (xx2)};
\draw[cyan] plot [smooth] coordinates {(x1) (4.7,-2.25) (5.6,-2.5) (7,-3.75) (xx2)};
\node[cyan] at (5,-2.65) {$R$};

\draw (7,-.5) circle (.5);
\draw (9,-1.5) circle (.5);
\draw (11,-1.5) circle (.5);
\draw (9.5,1) circle (.5);
\draw plot [smooth] coordinates {(4,1) (4.5,1.1) (5.5,1) (6,.5) (5.9,-.25) (6.5,-.5)};
\path (9.5,1) ++(-60:.5) coordinate (y1);
\path (11,-1.5) ++(105:.5) coordinate (y2);
\path (9.5,1) ++(-135:.5) coordinate (y3);
\path (7,-.5) ++(45:.5) coordinate (y4);

\draw plot [smooth] coordinates {(y1) (10,0.25) (10.5,-.25) (10.75,-.75) (y2)};
\draw plot [smooth] coordinates {(9.5,.5) (9.4,.2) (9.2,-.2) (9,-1)};
\draw plot [smooth] coordinates {(y3) (9,0.55) (8.5,.15) (7.5,0) (y4)};

\draw (1.5,-5) circle (.5);
\draw (-.75,-6) circle (.5);
\draw (1.5,-7) circle (.5);
\draw (.5,-8.75) circle (.5);

\path (-.75,-6) ++ (30:.5) coordinate (z1);
\draw plot [smooth] coordinates {(1,-3) (1.25,-3.75) (2.5,-4.25) (3.4,-5.75) (3,-6.5) (2,-7)};
\draw plot [smooth] coordinates {(-.25,-6) (.25,-6.1) (.75,-6.8) (1,-7)};
\draw plot [smooth] coordinates {(z1) (.25,-5.5) (.75,-5.1) (1,-5)};
\draw plot [smooth] coordinates {(-.75,-6.5) (-.65,-7) (0,-7.25) (-.5,-8.25) (-.5,-8.65) (0,-8.75)};
\draw plot [smooth] coordinates {(1,-8.75) (2.5,-8.9) (4.5,-8.1) (6.5,-8.5)};
\end{tikzpicture}
\end{center}
\caption{Connecting nakjis}
\end{figure}

Enumerate the vertices of $H$ as $x_1,\dots, x_{h}$ and edges of $H$ as $e_1=x_{s_1}x_{t_1},\dots, e_{h'}=x_{s_{h'}}x_{t_{h'}}$. Each of the nakjis guaranteed in Claim~\ref{cl-many-nakjis} corresponds to a vertex of $H$. We also give an ordering on each nakji's $\Delta$ legs: let $f_x$, $x\in V(H)$, be such that 
\[\left\{f_x(e): e\in E(H), x\in e\right\} = [d_H(x)].\]

Assume we have pairwise disjoint paths $R_1,\dots, R_{\ell}$ for some $0\leq \ell < h'$ satisfying the following for each $i\in [\ell]$, where $(i',j') = (s_{i},f_{s_{i}}(e_i))$ and $(i'',j'')=(t_i, f_{t_i}(e_i))$.
\stepcounter{propcounter}
\begin{enumerate}[label = {\bfseries \Alph{propcounter}\arabic{enumi}}]
	\item\label{R1} $R_i$ is a path between $v_{i',j'}$ and $v_{i'',j''} $  with length at most $10m$.

\item\label{R2} $R_i$ does not intersect with any $C_{i^*,0}$ for $i^*\in [h]$ and does not intersect with any $S_{i^*,j^*}\cup P_{i^*,j^*}$ for $(i^*,j^*)\in [h]\times[\Delta] \setminus \{(i',j'), (i'',j'') \}$.
\end{enumerate}
Let $(i^{\vdash},j^{\vdash})=(s_{\ell+1},f_{s_{\ell+1}}(e_{\ell+1}))$ and $(i^{\dashv},j^{\dashv})=(t_{\ell+1}, f_{t_{\ell+1}}(e_{\ell+1}))$. 
Let 
\[Y= \bigcup_{(i,j)\in [h]\times[\Delta] \setminus \{(i^{\vdash},j^{\vdash}), (i^{\dashv},j^{\dashv})\}} S_{i,j} \cup \bigcup_{i\in [h]} C_{i,0}.\]
Then by \ref{F3},~\ref{C2} and $d\le m^{100}$, we have 
$\abs{Y} \leq (\Delta +1) d \exp(\sqrt[50]{\log{n}})\leq \exp(\sqrt[49]{\log{n}}).$
Let  
\[R= \bigcup_{(i,j)\in [h]\times [\Delta]}V(P_{i,j})\cup \bigcup_{i\in [\ell]} V(R_i), \enspace \quad X=B^{\frac{1}{2}\sqrt{\log{n}}}_{G'-R}(S_{i^{\vdash},j^{\vdash}}) \enspace \quad \text{and} \quad\enspace X'=B^{\frac{1}{2}\sqrt{\log{n}}}_{G'-R}(S_{i^{\dashv},j^{\dashv}}).\]
As $\abs{R\cup L}\leq d m^2$, then by~\ref{C2} we can apply Proposition~\ref{prop: expansion after deleting shortest path} to $X$ and $X'$, with $R\cup L$ playing the role of $Y$ to obtain that 
\[\abs{X},\abs{X'} \geq \exp(\sqrt[10]{\log{n}}).\]
By \ref{F2} and~\ref{C1},~\ref{C2}, we know that $X,X'$ does not intersect with $Y$. As $\abs{Y\cup R\cup L} \leq 2\exp(\sqrt[49]{\log{n}}) <\frac{1}{4}\rho(\abs{Z})\abs{Z}$ for each $Z\in \{X,X'\}$, by Lemma~\ref{lem: path}, we can find a path $R$ from $X$ to $X'$ in $G'-R-Y$ with length at most $m$. As $S_{i,j}\subseteq B_{G'}^{\sqrt[10]{\log{n}}}(C_{i,j})$, we know $G[S_{i,j}]$ is connected due to~\ref{C2}, and $C_{i,j}$ is an $(\eps_1,\eps_2d)$-robust-expander with diameter at most $m$ due to~\ref{F1} and~\ref{C1}, using the definition of $X,X'$, we can extend $R$ inside $X,X'$ to get a path from $v_{i^{\vdash},j^{\vdash}}$ to $v_{i^{\dashv},j^{\dashv}}$ satisfying \ref{R1} and \ref{R2}. Repeating this for each $\ell=0,\dots, h'-1$, we obtain paths $R_1,\dots, R_{h'}$ satisfying \ref{R1} and \ref{R2}.

Note that for each $i\in [h']$ and $(i',j')=(s_i,f_{s_i}(e_i)), (i'',j'')=(t_i, f_{t_i}(e_i))$, the path $R_i$ and $P_{i',j'}\cup P_{i'',j''}$ might intersect at their interiors. However, $R_i\cup P_{i',j'}\cup P_{i'',j''}$ contains a path $P^*_{e_i}$ from $v_{i',j'}$ to $v_{i'',j''}$.

Moreover, \ref{C3} and \ref{R2} imply that each $P^*_{e_i}$ only intersects $C_{s_i,0}$ and $C_{t_i,0}$ at $u_{i',j'}$ and $u_{i'',j''}$ respectively, and  $P^*_{e_1},\dots, P^*_{e_{h'}}$ are all pairwise vertex disjoint.
Let $P^* = \bigcup_{i\in [h']} V(P^*_{e_{i}})$.

Now, for each $i\in [h]$, we consider $C_{i,0}$ and set $d_i= d_H(x_i)$.
Note that $C_{i,0}$ intersects with $P^*$ only at the distinct vertices $u_{i,1},\dots, u_{i,d_i}$.
We choose a vertex $u_{i,0} \in C_{i,0}\setminus\{u_{i,1},\dots, u_{i,d_i}\}$. As $G'[C_{i,0}]$ is $\eps^2 \nu d$-connected with $\eps^2 \nu d> \Delta$, we can find a subdivision of $K_{1,d_i}$ inside $C_{i,0}$, where $u_{i,0}$ corresponds to the centre of the star and the leaves correspond to $u_{i,1},\dots, u_{i,d_i}$.
These subdivisions of stars together with the paths $P^*_{e_i}$ all together yield a subdivision of $H$ in $G$. This provides a final contradiction and completes the proof.

\section{Concluding remarks}\label{sec: rmk}

\subsection{Bounded degree planar graphs as subdivision}
We might ask whether we can in fact guarantee subdivisions rather than minors in Theorem~\ref{planar}. Our methods for finding subdivisions only cover bounded-degree graphs, however. So a natural question might be: what is the best constant $c$ such that every graph of average degree $(c+o(1))t$ contains a subdivision of every bounded-degree planar graph of order $t$?

First, it is easy to see that we can achieve $c=5/2$: a planar graph with $t$ vertices has fewer than $3t$ edges, and any graph has a bipartite subgraph with at least half the edges, so after subdividing at most $3t/2$ edges we obtain a bipartite subdivision with at most $5t/2$ vertices. Theorem \ref{thm: main} then allows us to find a subdivision of this subdivision in any graph of average degree $(5/2+o(1))t$. In fact we can do better.
\begin{lemma}\label{lem:planar}Any planar graph $H$ on $t$ vertices has a bipartite subdivision with at most $2t-2$ vertices.
\end{lemma}
\begin{proof}We may assume $H$ is maximal planar, so every face is a triangle and there are $2t-4$ faces. Consider the dual graph $H^*$. This is a $2$-connected $3$-regular graph, and so, by Petersen's theorem, has a $1$-factor.
Each edge of this $1$-factor corresponds to an edge of the original graph $H$. We subdivide each of these edges once. Suppose there is an odd cycle $C$ in the original graph $H$. By double-counting edges bordering faces surrounded by $C$, there are an odd number of such faces. Thus the $1$-factor in $H^*$ contains an odd number of edges crossing $C$. Thus we have subdivided an odd number of edges of $C$, so $C$ becomes an even cycle in the subdivision. By exactly the same argument, we can see that any even cycle in $H$ remains even in the subdivision, so the subdivision is bipartite. It has exactly $\abs{H}+\abs{H^*}/2=2t-2$ vertices, as required.
\end{proof}
Lemma \ref{lem:planar} is best possible, since in any maximal planar graph we must subdivide at least one edge of each face, and this requires at least $t-2$ extra vertices. Together with Theorem \ref{thm: main}, it immediately gives the following.
\begin{prop}\label{planar-subdiv}For given $\eps>0$ and $\Delta\in \mathbb{N}$, there exists $d_0$ such that if $d\geq d_0$ and $H$ is a planar graph with at most $(1-\eps)d$ vertices and $\Delta(H)\leq \Delta$, and $G$ is a graph with average degree at least $2d$, then $G$ contains a subdivision of $H$.\end{prop}
Our lower bound from Theorem~\ref{planar} shows only that we cannot improve the constant $2$ in Proposition \ref{planar-subdiv} below $3/2$. However, no similar example will give a stronger lower bound, for the following reason. Any $t$-vertex planar graph $H$ has a bipartite subdivision with at most $3t/4$ vertices in one part (to see this, take a largest independent set $X$ and subdivide all edges which do not meet $X$; this is bipartite with one part being $V(H)\setminus X$, which has size at most $3t/4$ by the four-colour theorem), and this subdivision is a subgraph of $K_{3t/4,n}$ for sufficiently large $n$.
\begin{problem}What is the right value for the constant in Proposition \ref{planar-subdiv}? Is it $2$, $3/2$, or something in between?\end{problem}
\begin{problem}What can we say about $d_{\sfT}(H)$ for $k$-degenerate graphs $H$ for general $k\in\mathbb{N}$?\end{problem}

\subsection{Better bounds for minor closed families}

As $\alpha_2(F) \geq 2t/\chi(F)$, writing $\chi(\cF)=\max\{\chi(F):~F\in\cF\}$, Theorem~\ref{thm-minor-closed} implies the following Erd\H{o}s-Simonovits-Stone type bound for any minor-closed family $\cF$:
\begin{equation}\label{eq: ESS}
d_{\succ}(\cF,t)\leq 2\left(1-\frac{1}{\chi(\cF)}+o(1)\right)t.
\end{equation}
However, this is in general not tight. It is sharp for the disjoint union of cliques of order $\chi(\cF)$, if such a graph is in $\cF$. 

\begin{problem}
	Determine $d_{\succ}(\cF,t)$ for the nontrivial minor-closed family $\cF$. If $\cF$ is closed under disjoint union, do we have $d_{\succ}(\cF,t) = 2(1-1/\chi(\cF)+o(1))t$?
\end{problem}
In fact, the Hadwiger conjecture would imply $d_{\succ}(\cF,t) = 2(1-1/\chi(\cF)+o(1))t$ for $\cF$ closed under disjoint union.
To see this, consider the maximum $s$ such that $K_s\in \cF$. Then the Hadwidger conjecture would give $\chi(\cF) = s$. As the disjoint union of $K_s$ belongs to $\cF$, the above discussion implies $d_{\succ}(\cF,t) = 2(1-1/\chi(\cF)+o(1))t$. However, a minor-closed family does not have to be closed under disjoint union (for example, the class of graphs embeddable in a fixed surface other than the plane), and for such families we may have $d_{\succ}(\cF,t) < 2(1-1/\chi(\cF)-c)t$  for some absolute constant $c>0$, as we have seen in Theorem~\ref{planar}.

Another interesting question is to see when the upper and lower bounds arising from Theorem~\ref{thm-minor-closed} coincide, which motivates the following problem.
\begin{problem}
	For which graphs $G$ do we have $\abs{2\alpha(G) - \alpha_2(G)}=o(\abs{G})$? For which graphs $G$ do we have $\bigabs{\alpha(G)-\frac{\abs{G}}{\chi(G)}}=o(\abs{G})$? Do all minor closed families $\cF$ contain sufficiently large graphs with this property?
\end{problem}

The case $k=6$ of the Hadwiger conjecture, proved by Robertson, Seymour and Thomas \cite{RST93}, gives the above conclusion for any minor-closed family for which the minimal forbidden minors are all connected and include $K_6$. In particular, this applies to the linklessly embeddable graphs discussed in Section~\ref{sec-minor-closed}.
\begin{cor}The class $\mathcal L$ of linklessly embeddable graphs satisfies $d_{\succ}(\mathcal L,t)=(8/5+o(1))t$.
\end{cor}
Since the $Y\Delta Y$-reducible graphs form a subfamily of $\mathcal L$ containing the extremal example $rK_5$, the corollary also applies to this class.

\section*{Acknowledgement}
We thank O-joung Kwon for bringing our attention to the graph class of bounded expansion.

\medskip
\noindent
{\footnotesize
\begin{tabular}{@{}lll}
		John Haslegrave, Hong Liu	&\ & 	Jaehoon Kim        \\
		Mathematics Institute	 &\ & Department of Mathematical Sciences	  		 	 \\
		University of Warwick 	  &\ &  KAIST\\
		UK &\ & South Korea                             			 \\
\end{tabular}}

\begin{flushleft}
	\textit{Email addresses}:
		\texttt{$\lbrace$j.haslegrave,~h.liu.9$\rbrace$@warwick.ac.uk, jaehoon.kim@kaist.ac.kr.}
\end{flushleft}

\newpage

\appendix

\section{Proof of the robust expander lemma, Lemma~\ref{lem-expander}}
We first set up some functions as follows. For $x\ge 1$, define
\[\gamma(x)=C\int_{x}^{\infty} \frac{\rho(u)}{u}\,\mathrm{d}u.\]
Note that $\gamma(x)$ is a decreasing function. We will make use of the following two inequalities:
\begin{equation}\label{eq-gamma1}
\gamma(1)=\gamma(t/5)=C\int_{t/5}^{\infty} \frac{\eps_1}{u\log^2(15u/t)}\,\mathrm{d}u=\frac{C\eps_1}{\log 3}=\delta<1;
\end{equation}
and, for $C'>1$ and $x\ge t/2$, since $\rho(x)$ is decreasing and $\rho(x)\cdot x$ is increasing for $x\ge t/2$,
\begin{equation}
\gamma(x)-\gamma(C'x)=C\int_{x}^{C'x} \frac{\rho(u)}{u}\,\mathrm{d}u\ge C\rho(C'x) \int_{x}^{C'x} \frac{1}{u}\,\mathrm{d}u=C\log C'\cdot \rho(C'x)
\ge\frac{C\log C'}{C'} \rho(x).\label{eq-rho}
\end{equation}
Furthermore, set
\[\phi(G)=d(G)[1+\gamma(\abs{G})].\]
We say that $G$ is $\phi$-\emph{maximal} if $\phi(G)=\max_{H\subseteq G}\phi(H)$.
\begin{claim}\label{prop-min-deg}
	If $G$ is $\phi$-maximal, then $d(G)=\max_{H\subseteq G}d(H)$ and $\delta(G)\ge d(G)/2$.
\end{claim}
\begin{poc}
	Since $\gamma(x)$ is decreasing, for any $H\subseteq G$, by the $\phi$-maximality of $G$, we have that $d(G)[1+\gamma(\abs{G})]\ge d(H)[1+\gamma(\abs{H})]$
	and consequently $d(G)\ge d(H)$. Suppose there is a vertex $v$ with $d(v)<d(G)/2$. Let $H:=G-v$, then 
	\[d(H)=\frac{d(G)\abs{G}-2d(v)}{\abs{G}-1}>d(G),\]
	a contradiction.
\end{poc}

Take $H\subseteq G$ such that $H$ is $\phi$-maximal. We will show that $H$ is the desired robust expander. By Claim~\ref{prop-min-deg}, $\delta(H)\ge d(H)/2$. Since $H$ is $\phi$-maximal and that $\gamma(x)$ is decreasing, we have
\[d(H)\ge \frac{d(G)(1+\gamma(\abs{G}))}{1+\gamma(\abs{H})}\ge \frac{d(G)}{1+\gamma(1)}\ge (1-\gamma(1))d(G)\stackrel{\eqref{eq-gamma1}}{=}(1-\delta)d(G).\]
We are left to show that $H$ has the claimed vertex and edge expansions. 

Note that, since $H$ is $\phi$-maximal, for any $K\subseteq H$,
\begin{equation}\label{eq-phi-max}
	d(K)\le \frac{1+\gamma(\abs{H})}{1+\gamma(\abs{K})}\cdot d(H)\le d(H).
\end{equation}

Fix an arbitrary $X\subseteq V(H)$ with $t/2\le \abs{X}\le \abs{H}/2$,\footnote{Note that if $\abs{H}<t$, then the lemma is vacuously true.} and an arbitrary subgraph $F\subseteq H$ with $e(F)\le d(H)\cdot  \rho(\abs{X})\cdot \abs{X}$. Let $Y=X\cup N_{H\setminus F}(X)$ and $\xc=V(H)\setminus X$.

For the robust vertex expansion, note that
\begin{align*}
	&d(H)(\abs{X}+\abs{\xc})=2e(H)\le 2e(H[Y])+2e(H[\xc])+2e(F)\\
	&\stackrel{\eqref{eq-phi-max}}{\le} d(H[Y])\abs{Y}+d(H)\abs{\xc}+2d(H)\cdot \rho(\abs{X})\cdot \abs{X}\\
	\implies&d(H)\abs{X}\le d(H[Y])\abs{Y}+2d(H)\cdot \rho(\abs{Y})\cdot \abs{Y}\stackrel{\eqref{eq-phi-max}}{\le} \left(\frac{1+\gamma(\abs{H})}{1+\gamma(\abs{Y})} +2\rho(\abs{Y})\right) d(H)\abs{Y}\\ 
	\implies&  \frac{\abs{X}}{\abs{Y}}\le \frac{1+\gamma(H)}{1+\gamma(\abs{Y})}+2\rho(\abs{Y}).
\end{align*}
Thus,
\[\frac{\abs{N_{H\setminus F}(X)}}{\abs{Y}}=1-\frac{\abs{X}}{\abs{Y}}\ge \frac{\gamma(\abs{Y})-\gamma(\abs{H})}{1+\gamma(\abs{Y})}-2\rho(\abs{Y})\stackrel{\eqref{eq-gamma1}}{\ge} \frac{\gamma(\abs{Y})-\gamma(\abs{H})}{2}-2\rho(\abs{Y}).\]

Now, if $\abs{Y}\ge 3\abs{H}/4$, then $\abs{N_{H\setminus F}(X)}\ge \abs{H}/4\ge \abs{X}/2\ge \rho(\abs{X})\cdot \abs{X}$, yielding the desired robust vertex expansion. If $\abs{Y}\le 3\abs{H}/4$, then applying~\eqref{eq-rho} with $C'=4/3$, we have, as $\gamma(x)$ is decreasing and $C\ge 30$, that
\begin{align*}
	\frac{\abs{N_{H\setminus F}(X)}}{\abs{Y}}&\ge \frac{\gamma(\abs{Y})-\gamma(\abs{H})}{2}-2\rho(\abs{Y})\ge \frac{\gamma(\abs{Y})-\gamma(4\abs{Y}/3)}{2}-2\rho(\abs{Y})\\
	&\ge \frac{C\log(4/3)}{8/3}\cdot \rho(\abs{Y})-2\rho(\abs{Y})\ge \rho(\abs{Y}).
\end{align*}
Thus, $\abs{N_{H\setminus F}(X)}\ge \rho(\abs{Y})\cdot \abs{Y}\ge \rho(\abs{X})\cdot \abs{X}$.

The fact that $H$ is $\nu d$-connected was shown in~\cite[Lemma~5.2(iv)]{KLShS17}. 

This finishes the proof of Lemma~\ref{lem-expander}.

\section{Balanced homomorphism  to  an odd cycle: proof of Lemma \ref{lem: H partition odd}}
As $H$ has bandwidth at most $\beta d$, there exists an ordering $x_1,\dots, x_{\abs{H}}$ of $V(H)$ such that $x_ix_j\in E(H)$ implies $\abs{i-j}\leq \beta d$. Let $A$ and $B$ be a bipartition of $H$. Let $t= \lceil \beta^{-1/2}\rceil$  and we divide the vertices of $H$ according to the ordering as follows: for each $i \in [t^2]$ and $C\in \{A,B\}$, let 
	\[Y^C_i = \{ x_j \in C : (i-1)\beta d < j \leq i\beta d\} \enspace \quad \text{and} \quad \enspace Y_i=Y_i^A\cup Y_i^B.\]
	Note that this guarantees that no edge of $H$ is between $Y_i$ and $Y_{j}$ with $\abs{i-j}>1$. For each $i\in [t]$ and $C\in \{A,B\}$, let 
	\[W_i^C = \bigcup_{j\in [r]} Y_{(i-1)t+j}^C \enspace\quad \text{and}\quad  \enspace  Z_i^C = \bigcup_{j \in [t]\setminus[r]} Y_{(i-1)t+j}^C.\]
	In the claim below, we will decide to which part $X_j$ we assign the vertices in $Z_i^C$. To make such an assignment possible while keeping the edges only between two consecutive parts, we allow the vertices in $W_i^C$ to be assigned to some other parts. As each set $W_i^C$ is much smaller than $Z_i^C$, the uncontrolled assignments of $W_i^C$	will not harm us too much. 
Indeed, the set $W=\bigcup_{i\in [t]}(W_i^A\cup W_i^B)$ has size at most $r \beta^{1/2} d < \delta^2 d/r$. 	Now, we will decide to which part $X_j$ we will assign $Z_i^C$.
For this, we partition the set $[t]$ into $I_1,\dots, I_r$ as follows where $I_{0}=I_r$.
If $i\in I_{\ell}$, then we will later assign the vertices in $Z_i^A$ to $X_{\ell}$ and the vertices in $Z_{i}^B$ to $X_{\ell+1}$. Note that the vertices in the set $\bigcup_{i\in I_{\ell}} Z_i^A \cup \bigcup_{i\in I_{\ell-1}} Z_i^B$ below will be later assigned to $X_{\ell}$.
	
	\begin{claim}
		There exists a partition $I_1,\dots, I_r$ of $[t]$ satisfying the following.
		\begin{itemize}
			\item For each $\ell \in [r]$, we have $\abs{ \bigcup_{i\in I_{\ell}} Z_i^A \cup \bigcup_{i\in I_{\ell-1}} Z_i^B} \leq \frac{(1-\delta^2)d}{r} $.
		\end{itemize}	
	\end{claim}
	\begin{poc}
		We add each $\ell\in [t]$ independently to one of $I_1,\dots, I_r$ uniformly at random. Standard concentration inequalities (e.g.\ Azuma's inequality) easily show that with positive probability, we can ensure that for each 
		$\ell \in [r]$ we have
		\begin{align*}
		\Bigabs{ \bigcup_{i\in I_{\ell}} Z_i^A \cup \bigcup_{i\in I_{\ell-1}} Z_i^B}
		\leq \frac{\abs{A}}{r} + \frac{\abs{B}}{r} + \beta^{1/5}d \leq \frac{(1-\delta)d}{r} + \beta^{1/5}d  \leq \frac{(1-\delta^2)d}{r},	
		\end{align*}
		as $\beta \ll \delta, 1/r\ll 1$.
	\end{poc}

	With these sets $I_1,\dots, I_r$, we distribute the vertices in $Y_1,\dots, Y_{t^2}$ in order as follows. 
	
		Assume that we have distributed the vertices in $Y_1,\dots, Y_{s t}$ for some $0\le s<t$ among the sets $X_1,\dots, X_r$ as follows.
	\stepcounter{propcounter}
	\begin{enumerate}[label = {\bfseries \Alph{propcounter}\arabic{enumi}}]
		\item\label{one} For each $\ell \in [r]$ and $i\in I_{\ell}\cap [s]$, we have 
		$Z_i^A\subseteq X_\ell$ and $Z_{i}^B\subseteq X_{\ell+1}$.
		\item\label{two} Each edge of $H[\bigcup_{i\in [s t]} Y_i]$ lies between $X_\ell$ and $X_{\ell+1}$ for some $\ell\in [r]$.
		\item\label{three} $Y^B_{s t}\subseteq X_{\ell'}$ for some $\ell'\in [r]$.
	\end{enumerate}
	For $s=0$, the trivial partition $X_1=\dots X_r=\varnothing$ trivially satisfies the above (\ref{three} is vacuous as $Y_0$ does not exist. For technical reason, we assume $\ell'=r$ in this case). 
	
	We write $X_{r+i}=X_i$ for each $i\in [r]$.
	Now we describe how one can distribute vertices in $Y_{s t+1},\dots, Y_{(s+1)t}$.
	Let $s+1\in I_{\ell}$ for some $\ell\in [r]$. 
	Let $\ell^*$ be a nonnegative integer less than $2r$ such that $\ell'+\ell^* \equiv \ell ~({\rm mod}~r)$ and $\ell^*$ is odd. Such an $\ell^*$ exists as $r$ is odd.
	
	We put $Y^A_{s t+ 1}, Y^B_{s t+1} , Y^{A}_{s t+2} , Y^{B}_{s t+2},\dots, Y^B_{s t+ (\ell^*-1)/2}$ to $X_{\ell'+1}, X_{\ell'+2}, X_{\ell'+3}, X_{\ell'+4},\dots ,X_{\ell-1}$ respectively. Now we put all the remaining vertices in $Z^A_{s+1}\cup W^A_{s+1}$ to $X_{\ell}$ and all the remaining vertices in $Z^B_{s+1} \cup W^B_{s+1}$ to $X_{\ell+1}$.
	Then \ref{one}--\ref{three} hold for $s+1$. 
	
	By repeating this, we can distribute all vertices in $Y_1,\dots, Y_{t^2}$ into $X_1,\dots, X_r$ satisfying \ref{one}--\ref{three}.
	Then for each $\ell\in [r]$, we have 
	\[\abs{X_\ell} \leq \Bigabs{ \bigcup_{i\in I_{\ell}} Z_i^A \cup \bigcup_{i\in I_{\ell-1}} Z_i^B} + \abs{ W} \leq \frac{(1-\delta^2)d}{r} + r \beta^{1/2} d \leq \frac{d}{r}.\]
	This with \ref{two} proves the lemma.

\section{Webs with disjoint interiors: proof of Claim~\ref{cl: webs}}

In this section, we prove Claim~\ref{cl: webs}. We will use \eqref{eq: dm100}--\eqref{eq: G-L average degree} in Section~\ref{sec: medium}.
We first prove that the following holds.
\begin{equation} \label{eq: disjoint stars}
\begin{minipage}{0.9\textwidth}
For any set $X$ of size at most $d m^{90}$, the graph $G-L-X$ contains vertex disjoint stars $S_{1},\dots, S_{m^{85}}$ where each $S_{i}$ has exactly $2d/m^3$ leaves.
\end{minipage}	
\end{equation}
Indeed, if we have $S_{1},\dots, S_{\ell}$ for some $0\leq \ell< m^{85}$, then the set $X':=\bigcup_{i\in [\ell]} V(S_{i})$ has size at most $2dm^{82}$. By \eqref{eq: G-L average degree}, and the fact that $\Delta(G-L)\leq dm^{10}$, we have 
\[d(G-L-X-X') \geq \frac{d/m^2(n-\abs{L}) - dm^{10}(\abs{X}+\abs{X'}) }{ n-\abs{L}- \abs{X}-\abs{X'}} \geq \frac{  d n /(2m^{2}) - d n /m^{8} }{ n/2 } \geq \frac{2d}{m^3}.\]
The penultimate inequality holds as $n\geq d m^{100} \geq m^{10}\abs{X}$ by \eqref{eq: dm100} and $n-\abs{L}\geq n/2$. Thus $G-L-X$ contains a star $S_{\ell+1}$ with $2d/m^3$ leaves disjoint from $X'$.

Now we prove the following.
\begin{equation} \label{eq: disjoint units}
\begin{minipage}{0.9\textwidth}
For any set $X$ of size at most $d m^{30}$, the graph $G-L-X$ contains vertex disjoint $(m^{10}, d/m^3, m+2)$-units $F_{1},\dots, F_{m^{30}}$. \end{minipage}	
\end{equation}

Indeed, if we have $F_1,\dots, F_{\ell}$ for some $0\leq \ell < m^{30}$, then the set $X':=\bigcup_{i \in [\ell]} V(F_i)$ has size at most $m^{30}( 2d/m^3\cdot m^{10}) \leq d m^{38}$ vertices and $\abs{X\cup X'}< dm^{90}$. Hence \eqref{eq: disjoint stars} implies that there are vertex disjoint $2d/m^3$-stars $S_{v_1},\dots, S_{v_{m^{40}}}, S_{u_1},\dots, S_{u_{m^{80}}}$, with centres $v_1,\dots, v_{m^{40}}, u_1,\dots, u_{m^{80}}$ respectively. 

Let $P_1,\dots, P_{s}$ be pairwise internally disjoint paths of length at most $m+2$ with $s< m^{70}$ where $P_i$ is a $v_{c_i},u_{d_i}$-path for each $i\in [s]$ and $d_1,\dots, d_s$ are all distinct and each $P_i$ does not contain any of $\{v_1,\dots, v_{m^{40}}, u_1,\dots, u_{m^{80}}\}$ as an internal vertex.
Let 
\[V':= \bigcup_{i\in [m^{40}]} (V(S_{v_i})\setminus\{v_i\}) \enspace \text{and} \enspace U':= \bigcup_{i\in [m^{80}]-\{d_1,\dots, d_s\}} (V(S_{u_i})\setminus\{v_i\}),\]
then we have $\abs{V'}= 2dm^{37}$ and $\abs{U'}= (m^{80}-m^{70})\cdot (2d/m^3) > d m^{37}$. Further set
\[P':= \bigcup_{i\in [s]} (V(P_i)\setminus\{v_{c_i}\}), \enspace \text{and} \enspace U=\{v_1,\dots, v_{m^{40}}, u_1,\dots, u_{m^{80}}\}.\]
Then we have $\abs{X} + \abs{X'} + \abs{P'} + \abs{U}  \leq dm^{30} + dm^{38} + 2m \cdot m^{70}  + m^{40}+ m^{80}  \leq 2d m^{40} \leq \rho( dm^{37}) \cdot dm^{47}/4$.
Hence, applying Lemma~\ref{lem: path} with $V',U', X\cup X' \cup P'\cup U$ playing the roles of $X_1,X_2, W$ respectively, we can find a path of length at most $m$ between a vertex in $V(S_{v_{d_{s+1}}}) \setminus\{v_{s+1}\} \subseteq V'$ and a vertex in  $V(S_{u_{d_{s+1}}}) \setminus \{u_{d_{s+1}}\} \subseteq U'$ avoiding vertices in $X\cup X'\cup P'\cup U$. This  yields a $v_{d_{s+1}},u_{d_{s+1}}$-path $P_{s+1}$, which is internally disjoint from $X\cup X'\cup P'\cup U$ and $d_{s+1}\notin \{d_1,\dots, d_{s}\}$. Hence, this is internally disjoint from $P_1,\dots, P_{s}$ and $U$.

By repeating this for $s=0,1,\dots, m^{70}$, we obtain $P_1,\dots, P_{m^{70}}$. By pigeonhole principle, at least $m^{10}$ of $v_{c_1},\dots, v_{c_{m^{70}}}$ coincide, so there exists a vertex $v_j$ and $m^{10}$ internally disjoint paths from $v_j$ to $V(S_{u'_1}),\dots, V(S_{u'_{m^{10}}})$ for pairwise disjoint stars $S_{u'_1},\dots, S_{u'_{m^{10}}}$ where each star has $2d/m^3 \geq d/m^3+1$ leaves. These paths and stars together yields a $(m^{10},d/m^3,m+2)$-unit $F_{\ell+1}$ in $G-L$ disjoint from $X\cup X'$. Repeating this for $\ell=0,1\dots, m^{30}$, yields that the graph $G-L-X$ contains vertex disjoint $(m^{10}, d/m^3, m+2)$-units $F_{1},\dots, F_{m^{30}}$. This proves \eqref{eq: disjoint units}.

Now we are ready to prove Claim~\ref{cl: webs}.
Assume we have $(m^2,m^{10},d/m^3, 4m)$-webs $W_1,\dots, W_{\ell}$ with pairwise disjoint interiors, where $0\leq \ell < 2d$. Let $X = \bigcup_{i\in [\ell]} \Int(W_i)$, then
\begin{equation}\label{eq: X dm13}
\abs{X} \leq 2d \cdot m^2\cdot m^{10} \cdot 10m \leq 20 d m^{13}.
\end{equation}
Apply \eqref{eq: disjoint units} with the above set $X$ to obtain pairwise vertex disjoint $(m^{10}, d/m^3, m+2)$-units $F_{v_1},\dots, F_{v_{m^{12} }}$ and $F_{u_1},\dots, F_{u_{m^{25}}}$ in $G-L-X$, where the core vertex of $F_w$ is $w$ for each $w\in U$ and $U=\{v_1,\dots, v_{m^{12}}, u_1,\dots, u_{m^{25}}\}$.
Let $I$ be the union of all these units' interiors, then $\abs{I}\leq m^{40}$.

Similarly as before, let $P_1,\dots, P_{s}$ be pairwise internally disjoint paths each of length at most $4m$ with $0\leq s < m^{15}$ in $G-L-X$, where $P_i$ is a $v_{c_i},u_{d_i}$-path for each $i\in [s]$ and $d_1,\dots, d_s$ are all distinct. Moreover, 
we have $P_i = P(F_{v_{c_i}},w_i) \cup P'_i \cup P(F_{u_{d_i}}, z_i)$ with $w_i\in \Ext(F_{v_{c_i}}), z_i\in \Ext(F_{u_{d_i}})$ where $P'_i$ is an $w_i,z_i$-path with length at most $m$ disjoint from $I \cup U$.
Let
\[P':= \bigcup_{i\in [s]} \Int(P_{i}),\]
then we have $\abs{P'}\leq 4 m\cdot m^{15}\leq  m^{17}$.
Let 
\[U':= \bigcup_{i\in [m^{25}]: \Int(F_{u_i})\cap P'=\varnothing} \Ext(F_{u_i}).\]
then $\abs{U'}= (m^{25}-\abs{P'})\cdot (2d/m^3) > d m^{18}$. 
We may assume each $v_{c_i}$ is the endvertex of at most $m^2$ paths among $P_1,\ldots, P_s$ as otherwise these paths together with the  correpsonding units $F_{u_{d_i}}$ form a desired web in $G-L-X$. Thus, for each $i\in[m^{12}]$, letting $A_i\subseteq \Ext(F_{v_i})$ be the union of the leaves of the stars whose corresponding paths in $\Int(F_{v_i})$ is internally disjoint from $P_1,\ldots,P_s$ and $V'=\bigcup_{i\in [s]}A_i$, we have $\abs{V'}\ge m^{12}(m^{10}-m^2)d/m^3\ge dm^{18}$.

Then by \eqref{eq: X dm13}, we have $\abs{X}  + \abs{P'} + \abs{U}+\abs{I}  \leq dm^{13} + m^{17}  + m^{12}+ m^{25} +m^{40} \leq 2d m^{13} \leq \rho( dm^{18}) \cdot dm^{18}/4$.
Hence, applying Lemma~\ref{lem: path} with $V',U', X \cup P'\cup U\cup I$ playing the roles of $X_1,X_2, W$ respectively, we can find a $w_{s+1},z_{s+1}$-path path $P'_{s+1}$ of length at most $m$ with $w_{s+1}\in \Ext(F_{v_{c_{s+1}}}), z_{s+1}\in \Ext(F_{u_{d_{s+1}}})$ with $d_{s+1}\notin \{d_1,\dots, d_s\}$.
Also $P'_{s+1}$ avoids the vertices in $X\cup P'\cup U\cup I$, and $P_{s+1} = P(F_{v_{c_{s+1}}},w_{s+1}) \cup P'_{s+1} \cup P(F_{u_{d_{s+1}}}, z_{s+1})$ is of length at most $4m$ and  internally disjoint from $P_1,\dots, P_{s}$. By repeating this for $s=0,1,\dots, m^{15}$,  we obtain $P_1,\dots, P_{m^{15}}$. By the pigeonhole principle, at least $m^{2}$ of $v_{c_1},\dots, v_{c_{m^{15}}}$ coincide, so by relabelling, there exist $m^{2}$ internally disjoint paths $P_1,\dots, P_{m^2}$ from $v_1$ to $u_{1},\dots, u_{m^2}$, respectively in $G-L-X$, which yield an $(m^2,m^{10},d/m^3,4m)$-web $W_{\ell+1}$ which is disjoint from $L\cup X$. Repeating this for $\ell=0,\dots, 2d$, we obtain the desired webs as in Claim~\ref{cl: webs}, finishing the proof.
\end{document}